\documentclass{amsart}

\usepackage{latexsym}
\usepackage{amsfonts,amssymb} 
\usepackage{graphicx}
\usepackage{tikz}
\usepackage{caption}
\usepackage{subcaption}
\usepackage{ulem}
\newcommand{\field}[1]{\mathbb{#1}}
\newcommand{\A}{\field{A}}

\newcommand{\G}{\field{G}}

\newcommand{\N}{\field{N}}

\newcommand{\Z}{\field{Z}}

\newcommand{\krn}{{\rm ker}\,}

\theoremstyle{plain}

\newtheorem{theorem}{Theorem}[section]
\newtheorem{proposition}[theorem]{Proposition}
\newtheorem{lemma}[theorem]{Lemma}
\newtheorem{corollary}[theorem]{Corollary}
\newtheorem{definition}[theorem]{Definition}

\newtheorem{example}[theorem]{Example}

\newtheorem{question}[theorem]{Question}

\theoremstyle{definition}

\theoremstyle{remark}

\usepackage{fancyhdr}

\begin{document}

% Redefine "plain" pagestyle
\makeatletter	   % `@' is now a normal "letter' for LaTeX
\makeatother     % `@' is restored as a "non-letter" character

\title{Affine and Unirational unique factorial domains with unmixed gradings}
\author{Gene Freudenburg \and Takanori Nagamine}
\date{\today} 
\thanks{The work of the second author was supported by JSPS KAKENHI Grant Number JP21K13782.}
\maketitle
\pagestyle{fancy}
\lhead{}
\rhead{\it Affine and unirational UFDs with unmixed gradings}

%%%%%%%%%%%%%%%%%%%%%%%%%%%%%%%%%%%%%%%%%%%%%%%%%%%%%%%%%%%%%%%%%%%%%%%
\begin{abstract} This paper studies the class of unique factorial domains $B$ over an algebraically closed field $k$ which are affine or unirational over $k$ and which admit an effective unmixed $\Z^{d-1}$-grading with $B_0=k$, where $d$ is the dimension of $B$. 
Geometrically, these correspond to factorial affine $k$-varieties with an unmixed torus action of complexity one and trivial invariants.  
Our main result shows that this class is identical to the class of rings defined by trinomial data, thus generalizing earlier work of Mori, of Ishida, and of Hausen, Herrppich and S\"uss.
\end{abstract}
\maketitle

\section{Introduction} Let $k$ be a field and $G\cong\Z^n$ for some integer $n\ge 0$ and let $B$ be an integral domain over $k$ with $G$-grading $B=\bigoplus_{g\in G}B_g$. The grading is 
{\bf effective} if the weight monoid $M=\{ g\in\Z^n\, |\, B_g\ne 0\}$ generates $G$ as a group, and {\bf unmixed} if the maximal subgroup of $M$ is trivial. 
If $B$ is affine of dimension $n$ over $k$ and the grading is effective, then $B$ is the coordinate ring of an affine toric variety. These rings have been studied for a long time and can be described combinatorially using cones. 

This paper investigates unique factorization domains (UFDs) $B$ which are either affine or unirational over $k$, and admit an effective unmixed $\Z^{d-1}$-grading with $B_0=k$, where $d$ is the dimension of $B$ and $k$ is algebraically closed. 
Several authors have classified such rings subject to certain additional hypotheses. Mori \cite{Mori.77} and Ishida \cite{Ishida.77} classified them for dimensions $d=2,3$, respectively, in case $B$ is affine. Hausen, Herppich and S\"uss \cite{Hausen.Herppich.Suss.11} classified them for all dimensions in the case $k$ is algebraically closed of characteristic zero and $B$ is affine. 
Our main result is the classification of these rings with no additional hypotheses. 

To this end, assume that the following data $\Delta$ are given. 
\begin{itemize}
\item [($\Delta$.1)] An integer $n\ge 2$ and partition $n=n_0+\cdots +n_r$ where $r,n_i\ge 1$, $0\le i\le r$.
\item [($\Delta$.2)] A sequence $\beta_i=(\beta_{i1},\hdots ,\beta_{in_i})\in\Z_+^{n_i}$, $0\le i\le r$, where 
$\gcd (d_i,d_j)=1$ when $i\ne j$, $d_i=\gcd(\beta_{i1},\hdots ,\beta_{in_i})$. 
\item [($\Delta$.3)] A sequence of distinct elements $\lambda_2,\hdots ,\lambda_r\in k^*$.
\end{itemize}
Any such data set $\Delta$ will be called {\bf trinomial data} over $k$. 

Given trinomial data $\Delta$ over $k$, define the ring $k[\Delta ]$ as follows. Let $k^{[n]}=k[T_0,\hdots ,T_r]$ be the polynomial ring in $n$ variables over $k$, where $T_i=\{ t_{i1},\hdots ,t_{in_i}\}$, $0\le i\le r$, defines a partition of the set of variables
$t_{ij}$. Given $i$, let $T_i^{\beta_i}$ denote the monomial $t_{i1}^{\beta_{i1}}\cdots t_{in_i}^{\beta_{in_i}}$. We then define: 
\[
k[\Delta ]=k[T_0,\hdots ,T_r]/(T_0^{\beta_0}+\lambda_iT_1^{\beta_1}+T_i^{\beta_i})_{2\le i\le r}
\]
Observe the following.
\begin{enumerate}
\item If $r=1$, then $\lambda_i$ is the empty sequence and $k[\Delta ]=k^{[n]}$.  
\item If $K$ is an extension field of $k$, then $K\otimes_kk[\Delta ]=K[\Delta ]$. 
\end{enumerate}

Recently, the authors and Daigle showed the following.
\begin{theorem}\label{DFN} {\rm (\cite{Daigle.Freudenburg.Nagamine.22}, Theorem 5.1)} Let $k$ be any field. Given trinomial data $\Delta$ over $k$, $k[\Delta ]$ is an affine rational UFD of dimension $n-r+1\ge 2$ over $k$. 
\end{theorem}

Our main result builds on {\it Theorem\,\ref{DFN}}. 
\begin{theorem}\label{main} Suppose that $k$ is an algebraically closed field and $B$ is an integral domain of finite transcendence degree $d\ge 2$ over $k$. The following conditions for $B$ are equivalent.
\begin{enumerate}
\item $B\cong_kk[\Delta ]^{[m]}$ for some trinomial data $\Delta$ over $k$ and some $m\in\N$. 
\smallskip
\item $B$ is a unirational UFD which admits an effective unmixed $\Z^{d-1}$-grading with $B_0=k$. 
\smallskip
\item $B$ is an affine UFD which admits an effective unmixed $\Z^{d-1}$-grading with $B_0=k$. 
\smallskip
\item $B$ is a unirational UFD with $B^*=k^*$ which admits an effective $\Z^{d-1}$-grading with $B_0=k$.
\smallskip
\item $B$ is an affine UFD with $B^*=k^*$ which admits an effective $\Z^{d-1}$-grading with $B_0=k$.
\end{enumerate}
\end{theorem}
{\it Theorem\,\ref{DFN}} combined with {\it Proposition\,\ref{tri-data}} shows that if condition (1) holds, then conditions (2)-(5) also hold. 
Our main result, {\it Theorem\,\ref{classify}}, shows that, if condition (2) holds, then (1) holds; the theory of signature sequences is developed in {\it Section 3} to prove this theorem. 
{\it Corollary\,\ref{affineUFD}} shows that, if condition (3) holds, then condition (2) holds. 
To complete the cycle of equivalences, {\it Proposition\,\ref{Laurent}(c)} shows that condition (4) implies condition (2), and that condition (5) implies condition (3).

In \cite{Mori.77}, Mori classified rings satisfying condition (3) in the case $d=2$. 
These rings are of the form $B=k[\Delta ]$ for trinomial data $\Delta$ using the unit partition $n=1+\cdots +1$ for some integer $n\ge 2$. 
In {\it Section 6.1} we recover Mori's description from {\it Theorem\,\ref{main}}. 
In particular, we show that $k[\Delta ]\cong_kk[\Delta^{\prime}]$ for trinomial data $\Delta^{\prime}$ in Mori form (see {\it Section 6} for the definition), and that 
when $k$ is of characteristic zero, the presentation of $B$ by trinomial data in Mori form is unique,
i.e., if $k[\Delta ]\cong_kk[\Delta^{\prime}]$ and $\Delta ,\Delta^{\prime}$ are in Mori form, then $\Delta=\Delta^{\prime}$ ({\it Theorem\,\ref{non-iso}}). 

In \cite{Ishida.77}, Ishida classified rings satisfying condition (3) in the case $d=3$. 
These rings are of the form $B=k[\Delta ]$ for trinomial data $\Delta$ using a partition of the form $n=1+\cdots +1+2$ for some integer $n\ge 2$. 
In {\it Section 6.2} we recover Ishida's description from {\it Theorem\,\ref{main}}. 

In \cite{Hausen.Herppich.Suss.11}, Hausen, Herppich and S\"uss classified rings satisfying condition (3) in the case $k$ is of 
characteristic zero.\footnote{The cited paper does not include the unmixed hypothesis in its assumptions, 
but it is clear that this is a tacit assumption without which the paper's main results would not be correct.} 
They define data $(A,\mathfrak{n},L)$ which they call an admissible triple. 
The associated ring $R(A,\mathfrak{n},L)$ is a normal affine domain of dimension $d$ with an effective $\Z^{d-1}$-grading with $B_0=k$. 
Theorem 1.9 specifies which of these are UFDs and shows that all rings satsifying condition (3) are of this form. Subsequently, Wrobel \cite{Wrobel.20} described the divisor class group of $R(A,\mathfrak{n},L)$.

Note that condition (2) does not assume that the ring $B$ is affine. One source of interest in these rings comes from invariant theory for unipotent groups acting on affine $k$-varieties, where rings of invariants are not, in general, noetherian. 
Applications of {\it Theorem\,\ref{main}} to unipotent actions are given in {\it Section\,\ref{apps}}. In \cite{Knop.93}, Knop showed that a normal unirational variety over an algebraically closed field which admits a reductive group action of complexity one must be an affine variety. Our proof of {\it Theorem\,\ref{main}} relies on Knop's result. 

When $k$ is algebraically closed, {\it Corollary\,\ref{smooth}} shows that the only smooth varieties among those of the form $X={\rm Spec}(k[\Delta ])$ for trinomial data $\Delta$ are the affine spaces. 

\medskip

\noindent {\bf Preliminaries.} For a ground field $k$, a {\bf $k$-domain} $B$ is an integral domain containing $k$. $B^*$ is the group of units of $B$, ${\rm frac}(B)$ is the quotient field of $B$, and for the integer $n\ge 0$, 
$B^{[n]}$ is the polynomial ring in $n$ variables over $B$ and $B^{[\pm n]}$ is the ring of Laurent polynomials in $n$ variables over $B$. When $B$ is a Krull domain, $\dim B$ denotes its Krull dimension. 
If $K$ is a field, then $K^{(n)}$ denotes the field of fractions of $K^{[n]}$. 
Affine $n$-space over $k$ is denoted by $\A^n_k$, $\G_a$ is the additive group of $k$, and $\G_m$ the multiplicative group of units of $k$. 
The {\bf Makar-Limanov invariant} $ML(B)$ of $B$ is the intersection of all invariant rings of $\G_a$-actions on $B$.
$B$ is {\bf rigid} if $ML(B)=B$, and {\bf stably rigid} if $ML(B^{[n]})=B$ for every $n\ge 0$. In case the characteristic of $k$ is 0, ${\rm LND}(B)$ denotes the set of locally nilpotent derivations of $B$. 
If $\mathfrak{g}$ is a grading of $B$ by some abelian group, then ${\rm LND}(B,\mathfrak{g})$ is the subset of $\mathfrak{g}$-homeneous elements of ${\rm LND}(B)$.
See \cite{Freudenburg.17} for details about locally nilpotent derivations and $\G_a$-actions. 

\medskip

\noindent {\bf Acknowledgments.} The authors wish to thank Dayan Liu and Xiaosong Sun of Jilin University, Daniel Daigle of the University of Ottawa, Ivan Arzhantsev of the Higher School of Economics, Moscow, and Adrien Dubouloz of the Universite de Bourgogne for their helpful comments regarding this paper. 

The work of the second author was supported by JSPS KAKENHI Grant Number JP21K13782.
%%%%%%%%%%%%%%%%%%%%%%%%%%%%%%%%%%%%%%%%%%%%%%%%%%%%%%%%%%%%%%%%%%%%%%%

\section{Degree Functions and Gradings}\label{degree+grading}

\subsection{Monoids} This paper considers gradings by the groups $G=\Z^n$ for $n\ge 0$, and we restrict our attention to submonoids $M$ of these groups. As such, $M$ is both cancelative and torsion-free. We say that $M$ is {\bf unmixed} if its maximal subgroup $\{ g\in M\, |\, -g\in M\}\subset M$ is trivial; see \cite{Kambayashi.Russell.82}, 1.5. 
Note that any submonoid of an unmixed monoid is unmixed. 

Recall that $(G,<)$ is a {\bf totally ordered abelian group} if $<$ is a total order of $G$ such that:
\begin{center}
For all $x,y,z\in G$, $x+z< y+z$ implies $x< y$.
\end{center}
$G$ always admits such an order, for example, lexicographical order relative to a $\Z$-basis of $G$. 
Define the monoid $G_+=\{ g\in G\, |\, 0\le g\}$, which is unmixed.

\begin{proposition}\label{unmixed-monoid} Let $G=\Z^n$ and $M\subset G$ a finitely generated submonoid. The following conditions for $M$ are equivalent.
\begin{enumerate}
\item $M$ is unmixed. 
\item There exists a $\Z$-basis $\{u_1,\hdots ,u_n\}$ of $G$ such that $M\subset \N u_1\oplus\cdots\oplus\N u_n$.
\item There exists a total order $<$ on $G$ such that $(G,<)$ is a totally ordered abelian group and $(M,<)$ is a well-ordered set.  
\item There exists a total order $<$ on $G$ such that $(G,<)$ is a totally ordered abelian group and $M\subset G_+$. 
\end{enumerate}
\end{proposition}

\begin{proof} That condition (1) implies condition (2) is \cite{Kambayashi.Russell.82}, Lemma 1.6.

Assume that condition (2) holds and let $\{ u_1,\hdots ,u_n\}$ be a $\Z$-basis of $G$ such that $M\subset \bigoplus_{1\le i\le n}\N u_i$. Let $<$ be lexicographical order for $G$ relative to this basis. Then 
$(G,<)$ is a totally ordered abelian group and $\N u_1\oplus\cdots\oplus\N u_n$ is a well-ordered set for $<$. Since $M\subset \N u_1\oplus\cdots\oplus\N u_n$, we see that
$(M,<)$ is a well-ordered set. So condition (3) holds. 

Assume that condition (3) holds. Let $<$ be a total order on $G$ such that $(G,<)$ is a totally ordered abelian group and $(M,<)$ is a well-ordered set. Let $\mu =\min M$. If $\mu\le 0$ then 
$2\mu =\mu +\mu\le \mu+0=\mu$. Since $2\mu\in M$ we also have $\mu\le 2\mu$. So $2\mu =\mu$, which gives $\mu=0$. Therefore, $M\subset G_+$ and condition (4) holds.

Assume that condition (4) holds. Let $<$ be a total order on $G$ such that $(G,<)$ is a totally ordered abelian group and $M\subset G_+$. Since $G_+$ is unmixed, $M$ is also unmixed. So condition (1) holds.
\end{proof}

The reader should note that the total order in conditions (3) and (4) of this proposition is not an arbitrary total order on $G$. For example, 
if $G\cong\Z^n$ and $(G,<)$ is a totally ordered abelian group, then $\{ -g\, |\, g\in G_+\}$ is an unmixed monoid having no least element. 

\subsection{Degree Functions} 
Let $(G,<)$ be a totally ordered abelian group for $G\cong\Z^n$ and let $B$ be a $k$-domain with degree function $\deg :B\to G\cup\{ -\infty\}$. 
{\it We will always assume that nonzero elements of the ground field $k$ are of degree 0.} The set
\[
 \deg (B)=\{ \deg (b)\, |\, b\in B\, ,\, b\ne 0\}\subset G
 \]
is a monoid, called the {\bf degree monoid} for $\deg$. The induced filtration of $B$ is
\[
B=\bigcup_{g\in G}\mathcal{F}_g
\]
where the sets $\mathcal{F}_g=\{ b\in B\, |\, \deg b\le g\}$ are the associated {\bf degree modules}. The associated {\bf degree submodules} are:
\[
\mathcal{V}_g=\{ f\in B\,\vert\, \deg f<g\}\subset\mathcal{F}_g
\]
Each degree module $\mathcal{F}_g$ is a $k$-vector space, and the associated degree submodule $\mathcal{V}_g$ is a subspace of $\mathcal{F}_g$. 
The degree function $\deg$ can be extended to $K={\rm frac}(B)$ by letting $\deg (f/g)=\deg f-\deg g$ for $f,g\in B$, $g\ne 0$. 
Note that, if $B$ is a field, then $\deg$ is a degree function on $B$ if and only if $(-\deg)$ is a discrete valuation of $B$. 

Recall that a subalgebra $A\subset B$ is {\bf factorially closed} in $B$ if $rs\in A$ for nonzero $r,s\in B$ implies $r\in A$ and $s\in A$.

\begin{proposition}\label{degree} With the assumptions and notation above:
\begin{itemize}
\item [{\bf (a)}] $\mathcal{F}_0$ is a subring of $B$ which is integrally closed in $B$.
\item [{\bf (b)}] $\mathcal{F}_d$ is an ideal of $\mathcal{F}_0$ for each $d\le 0$. 
\item [{\bf (c)}] If $B$ is a normal ring then $\mathcal{F}_0$ is a normal ring. 
\item [{\bf (d)}] If $B$ is a field then $\mathcal{F}_0$ is a valuation ring of $B$ and ${\rm frac}(\mathcal{F}_0)=B$. 
\item [{\bf (e)}] If $\deg (B)\subset G_+$ then $\mathcal{F}_0$ is factorially closed in $B$ and $B^*\subset\mathcal{F}_0$. 
\item [{\bf (f)}]  If $B$ is a UFD and $\deg (B)\subset G_+$, then $\mathcal{F}_0$ is a UFD. 
\end{itemize}
\end{proposition} 

\begin{proof} Extend $\deg$ to $K$ and let $V=\{ f\in K\,\vert\, \deg f\le 0\}$. Then $V$ is a valuation ring of $K$, and $\mathcal{F}_0=V\cap B$. This implies parts (a), (c) and (d), and part (b) is clear. 

For part (e), assume that $\deg (B)\subset G_+$. Then $\mathcal{F}_0=B_0$ and, since $\deg (B)$ is a submonoid of an unmixed monoid $G_+$, $\deg (B)$ is unmixed. If $ab\in B_0$ for nonzero $a,b\in B$, then $0=\deg (ab)=\deg a+\deg b$ in $\deg (B)$ implies $\deg a=\deg b=0$, so $a,b\in B_0$. 
Therefore, $B_0$ is factorially closed in $B$. 

Part (f) is implied by part (e); see \cite{Freudenburg.17}, Lemma 2.8. 
\end{proof}

We say that $\deg$ is of {\bf finite type} if $\dim_k\mathcal{F}_g<\infty$ for each $g\in G$. 

\begin{lemma}\label{finite-type} If $\deg$ is of finite type, then $\mathcal{F}_g=\{ 0\}$ for each $g<0$ and $\mathcal{F}_0=k$.
\end{lemma}

\begin{proof} Given $f\in\mathcal{F}_g$ for $g<0$, we have:
\[
fk[f]\subset\mathcal{F}_g \implies \dim_k fk[f]<\infty\implies \dim_k k[f]<\infty 
\]
We conclude that $k[f]$ is a field. If $f\ne 0$, then $f\in k[f]^*$. But then
\[
f^{-1}\in k[f]\subset\mathcal{F}_0 \quad {\rm and}\quad \deg f^{-1}>0
\]
which is a contradiction. Therefore, $f=0$ and $\mathcal{F}_g=\{ 0\}$. 

By {\it Proposition\,\ref{degree}(a)}, $\mathcal{F}_0$ is a $k$-subalgebra of $B$, and by hypothesis, it is also a finite-dimensional vector space over $k$. Therefore, $\mathcal{F}_0$ is a finite field extension of $k$. 
By what was shown above, $\deg (B)\subset G_+$. By {\it Proposition\,\ref{degree}(e)}, $\mathcal{F}_0$ is factorially closed, hence algebraically closed, in $B$. Therefore, $\mathcal{F}_0=k$. 
\end{proof}

\subsection{Gradings} 
Let $B$ be a $k$-domain and $G\cong\Z^n$ for some $n\ge 0$, and let 
$\mathfrak{g}$ be a $G$-grading of $B$ given by:
\[
B=\bigoplus_{g\in G}B_g
\]
{\it We will always assume that $k\subset B_0$.}  

Each set $B_g$ is a vector space over $k$ and $B_0$ is a subring. 
Given $b\in B$ we have $b=\sum_{g\in G}b_g$ where $b_g=\pi_g(b)$ for projections $\pi_g :B\to B_g$. 
The {\bf support} of $b\in B$ is the finite set 
\[
{\rm Supp}(b)=\{ g\in G\, |\, b_g\ne 0\}\subset G
\]
which is empty if and only if $b=0$. 
Note that, since $G$ is torsion-free, whenever $fg$ is homogeneous for nonzero $f,g\in B$, we have that $f$ and $g$ are homogeneous. In particular, this means that every $f\in B^*$ is homogeneous. 

Since $B$ is a $k$-domain, the set $M_{\mathfrak{g}}(B)=\{g\in G \ | \ B_g\not=0 \}$ is a submonoid of $G$, 
called the {\bf weight monoid} of $\mathfrak{g}$.
We define the following terms. 
\begin{itemize}
\item [(a)] $\mathfrak{g}$ is of {\bf finite type} if $\dim_kB_g<\infty$ for each $g\in G$.
\item [(b)] $\mathfrak{g}$ is {\bf effective} if $M_{\mathfrak{g}}(B)$ generates $G$ as a group. 
\item [(c)] $\mathfrak{g}$ is {\bf unmixed} if  $M_{\mathfrak{g}}(B)$ is an unmixed monoid. 
\end{itemize}

Any total order $<$ of $G$ defines a degree function $\deg_{\mathfrak{g}}:B\to G\cup\{ -\infty\}$ by $\deg_{\mathfrak{g}}(b)=\max{\rm Supp}(b)$.
Assume that $(G,<)$ is a totally ordered abelian group. 
If $b\in B$ is nonzero, define the {\bf highest-degree homogenous summand} of $b$ by $\bar{b}=b_g$ where $g=\deg_{\mathfrak{g}}(b)$. Observe the following.
\begin{enumerate}
\item $\deg_{\mathfrak{g}}(B)=M_{\mathfrak{g}}(B)$ regardless of the total order used.
\item If $\deg_{\mathfrak{g}}$ is of finite type then $\mathfrak{g}$ is of finite type and unmixed.
\end{enumerate}
Statement (1) and the first part of statement (2) follow from the observation that $B_g\subset\mathcal{F}_g$ for each $g\in G$. The second part of statement (2) follows from {\it Lemma\,\ref{finite-type}}. 
Note that the converse of the first part of statement (2) may fail,
for example, if $B=k[x,x^{-1}]$, the ring of Laurent polynomials with the standard $\Z$-grading, 
then the grading is of finite type, but the associated degree function is not.  

The following result generalizes Lemma 2.10 of \cite{Daigle.Freudenburg.Nagamine.22}. 
\begin{proposition}\label{Laurent}
Let $B$ be a $k$-domain with an effective $\Z^n$-grading for some integer $n\ge 1$.
\begin{itemize}
\item [{\bf (a)}] There exist nonzero homogeneous $f_1,\hdots ,f_n\in B$ and $w_1,\hdots ,w_n\in S^{-1}B$ such that
\[
S^{-1}B=(S^{-1}B)_0[w_1,w_1^{-1},\hdots ,w_n,w_n^{-1}]\cong (S^{-1}B)_0^{[\pm n]}
\]
where $S\subset B$ is the multiplicatively closed set generated by $f_1,\hdots ,f_n$.
\item [{\bf (b)}] Let $K$ be the field:
\[
 K=\{ u/v\in {\rm frac}(B) \,\vert\, u,v\in B_g , v\ne 0, g\in G\}
 \]
 Then $K={\rm frac}((S^{-1}B)_0)$ and ${\rm frac}(B)\cong_K K^{(n)}$.
 \item [{\bf (c)}] If $B_0$ is a field and $B^*=B_0^*$ then the grading is unmixed. 
 \end{itemize}
\end{proposition}

\begin{proof} Let $\{ u_1,\hdots ,u_n\}$ be a $\Z$-basis of $G:=\Z^n$. 
Since the grading is effective, for each $1\le i\le n$, there exist nonzero homogeneous $a_i,b_i\in B$ such that $\deg(a_i)-\deg(b_i)=u_i$. Define $f_i\in B$ by $f_i=a_ib_i$, $1\le i\le n$, and let $S\subset B$ be the multiplicatively closed set generated by $f_1,\hdots ,f_n$. Since $S$ is generated by homogeneous elements, the $\Z^n$-grading of $B$ extends to a $\Z^n$-grading of $S^{-1}B$. Define $w_i\in (S^{-1}B)^*$ by $w_i=a_i/b_i$, noting that $w_i$ is homogeneous of degree $u_i$.

Given nonzero homogeneous $x\in S^{-1}B$, write $\deg (x)=\sum_{i=1}^ne_iu_i$ for $e_i\in\Z$. Then:
\[
\deg (xw_1^{-e_1}\cdots w_n^{-e_n})=0 \implies x\in (S^{-1}B)_0[w_1,w_1^{-1},\hdots ,w_n,w_n^{-1}]
\]
Therefore, $S^{-1}B=(S^{-1}B)_0[w_1,w_1^{-1},\hdots ,w_n,w_n^{-1}]$. Since $\deg (w_1),\hdots ,\deg (w_n)$ are $\Z$-linearly independent in $G$, it follows that:
\[
S^{-1}B=(S^{-1}B)_0[w_1,w_1^{-1},\hdots ,w_n,w_n^{-1}]\cong (S^{-1}B)_0^{[\pm n]}
\]
This proves part (a).

For part (b), note that part (a) implies ${\rm frac}(B)=L(w_1,\hdots ,w_n)$ where $L={\rm frac}((S^{-1}B)_0)$. Since $(S^{-1}B)_0\subset K$ we have $L\subset K$. Conversely, given $g\in G$, write $g=e_1u_1+\cdots +e_nu_n$ for integers $e_i$.  
Given nonzero $u\in B_g$ define $u'\in (S^{-1}B)_0$ by:
\[
u'=\frac{u(b_1^{e_1}\cdots b_n^{e_n})^2}{f_1^{e_1}\cdots f_n^{e_n}}
\]
Then for nonzero $u,v\in B_g$ we see that $u/v=u'/v'\in {\rm frac}((S^{-1}B)_0)$ and $K\subset L$.

For part (c), let $H\subseteq G$ be the maximal subgroup of $\deg (G)$. Assume that there exists nonzero $g\in H$. Then there exist nonzero $p\in B_g$ and $q\in B_{-g}$. 
But then $pq\in B_0\setminus\{ 0\}\subset B^*$ implies $p\in B^*=B_0^*$, which gives a contradiction. Therefore, $H=\{ 0\}$. 
\end{proof}

\begin{corollary}\label{unirational}  Let $B$ be a unirational $k$-domain of finite transcendence degree $d\ge 1$ over $k$ with an effective $\Z^n$-grading for some $n\le d$ and let $K$ be as in {\it Proposition\,\ref{Laurent}}.
If either (i) $n=d-1$, or (ii) $n=d-2$ and $k$ is algebraically closed of characteristic zero, then $K\cong_kk^{(d-n)}$ and $B$ is rational over $k$.
\end{corollary}

\begin{proof} By {\it Proposition\,\ref{Laurent}}, ${\rm frac}(B)\cong_KK^{(n)}$, 
which implies ${\rm tr.deg}_kK=d-n$. By hypothesis, ${\rm frac}(B)\subset k^{(r)}$ for some $r\ge d$. 

Assume $n=d-1$. Then ${\rm tr.deg}_kK=1$ and L\"uroth's Theorem\index{L\"uroth's Theorem} implies that $K \cong k^{(1)}$.

Assume $n=d-2$ and $k$ is algebraically closed of characteristic zero. Then ${\rm tr.deg}_kK=2$ and Castelnuovo's Theorem \cite{Castelnuovo.1894} implies that $K=k (u,v)\cong k^{(2)}$ for some $u,v\in K$. 

In either case, {\it Proposition\,\ref{Laurent}} implies ${\rm frac}(B)\cong k^{(d)}$. 
\end{proof}

\begin{corollary}\label{affineUFD}  Assume that $k$ is algebraically closed. Let $B$ be an affine UFD of dimension $d\ge 1$ over $k$ with an effective $\Z^{d-1}$-grading. 
Then $B$ is rational over $k$.
\end{corollary}

\begin{proof}  By {\it Proposition\,\ref{Laurent}}, there exist nonzero homogeneous $f_1,\hdots ,f_{d-1}\in B$ such that, if $S\subset B$ is the multiplicatively closed set generated by $f_1,\hdots ,f_{d-1}$, then 
$S^{-1}B\cong R^{[\pm (d-1)]}$ where $R=(S^{-1}B)_0$. $R$ has the following properties. 
\begin{enumerate}
\item $R$ is a UFD: Since $B$ is a UFD, $S^{-1}B$ is also a UFD. Since $S^{-1}B\cong R^{[\pm (d-1)]}$ we see that $R$ must be a UFD.  
\smallskip
\item $R$ is $k$-affine: $S^{-1}B=B[f_1^{-1},\hdots ,f_{d-1}^{-1}]$ and since $B$ is affine over $k$, $S^{-1}B$ is affine over $k$. Since $S^{-1}B\cong R^{[\pm (d-1)]}$ it follows that $R$ is affine over $k$. 
\smallskip
\item $\dim R=1$: This is clear from $S^{-1}B\cong R^{[\pm (d-1)]}$ and $d=\dim B$. 
\end{enumerate}
By \cite{Freudenburg.17}, Lemma 2.9, any ring satisfying these three conditions over an algebraically closed field is a localization $k[x]_f$ of a polynomial ring $k[x]\cong k^{[1]}$ for some $f\in k[x]$. 
Therefore, ${\rm frac}(R)\cong k(x)\cong k^{(1)}$ and ${\rm frac}(B)\cong k^{(d)}$. 
\end{proof}

This section concludes with the aforementioned theorem of Knop.
\begin{theorem}\label{Knop} \cite{Knop.93} Let $k$ be an algebraically closed field and $G$ a reductive group over $k$. If $X$ is a normal unirational $G$-variety of complexity at most one, 
then the algebra $k[X]$ of regular functions is finitely generated.
\end{theorem}

The reader is referred to Knop's article for the general definition of complexity. 
For the purposes of this article, it suffices to know that, when $B$ is a unirational UFD of transcendence degree $n$ over $k$, an effective $\Z^{n-1}$-grading of $B$ corresponds a complexity one action of the torus $T=\G_m^{n-1}$ on the variety defined by $B$, where $B^T=B_0$. 

%%%%%%%%%%%%%%%%%%%%%%%%%%%%%%%%%%%%%%%%%%%%%%%%%%%%%%%%%%%%%%%%%%%%%%%%%%%%%%%%%%%

\section{Signature Sequences}\label{signature}
\subsection{Definition and Basic Properties}
Let $k$ be a field, $B$ a $k$-domain and $(G,<)$ a finitely generated totally ordered abelian group.
In this section, we consider pairs $(B,\deg )$ where 
\[
\deg :B\to G\cup\{ -\infty\}
\]
is a degree function such that $\deg (B)\subset G_+$. 
By {\it Proposition\,\ref{unmixed-monoid}}, $\deg (B)$ is an unmixed monoid, and by {\it Proposition\,\ref{degree}(e)}, 
$\mathcal{F}_0$ is factorially closed in $B$ and $B^*\subset\mathcal{F}_0$. 
Let 
\[
B=\bigcup_{g\in G}\mathcal{F}_g
\]
be the induced filtration of $B$ by degree modules. 
\begin{definition}\label{sig-seq} 
A {\bf signature sequence} $\vec{h}=\{ h_i\}_{i\in I}$ for $(B,\deg )$ is a sequence $h_i\in B$ indexed by an interval $0\in I\subset \N$ such that:
\begin{enumerate}
\item $h_0=1$ and $d_0=\deg h_0=0$.
\item For each $n\in I$ with $n\ge 1$, $h_n\in \mathcal{F}_{d_n}\setminus k[h_1,\hdots ,h_{n-1}]$ where:
\[
d_n=\min\{ g\in G\, |\, \mathcal{F}_g\not\subset k[h_1,\hdots , h_{n-1}]\}
\]
\end{enumerate}
The {\bf length} of $\vec{h}$ is the cardinality of $I$, denoted $\vert\vec{h}\vert$, and $\vec{h}$ is {\bf finite} or {\bf infinite} depending on $\vert\vec{h}\vert$. $\vec{h}$ is 
{\bf complete} if $B=k[\vec{h}]$.
\end{definition}
Note that, for $n\le \vert\vec{h}\vert$, the subsequence $\{ h_0,\hdots ,h_n\}$ is a signature sequence. 
In addition, the degree sequence $\{d_n\}\subset G$ has $d_i\le d_{i+1}$ whenever $i,i+1\in I$. 

\begin{lemma}\label{independent} 
If $\vec{h}=\{ h_i\}_{i\in I}$ is a signature sequence for $(B,\deg )$, then $\{ h_i\}_{i\in I}$ is a linearly independent set over $k$.
\end{lemma}

\begin{proof}
Assume that $\sum_{0\le i\le m}c_ih_i=0$ for $c_i\in k$ and some $m\ge 1$. If $c_m\ne 0$ then $h_m\in k[h_1,\hdots ,h_{m-1}]$, which is not the case, meaning that $c_m=0$. By induction, $c_i=0$ for each $i$. 
\end{proof}

\begin{lemma}\label{complete} If $\deg$ is of finite type, then $(B,\deg )$ admits a complete signature sequence. 
\end{lemma}

\begin{proof} There are two cases to consider.

Case 1: There exists a complete finite signature sequence for $(B,\deg )$. 

Case 2: There is no complete finite signature sequence for $(B,\deg )$. In this case, 
any finite signature sequence $\{ h_0,\hdots ,h_n\}$ can be extended, that is, 
\[
d:=\min\{ g\in G\,\vert\, \mathcal{F}_g\not\subset k[h_1,\hdots ,h_n]\}
\]
exists, and we can choose $h_{n+1}\in \mathcal{F}_d\setminus k[h_1,\hdots ,h_n]$. By induction, there exists an infinite signature sequence $\vec{h}$. 
Since $\dim_k\mathcal{F}_g<\infty$ for each $g$, it follows that, given $g\in G$:
\[
\mathcal{F}_g\subset k[h_1,\hdots ,h_n] \quad {\rm for} \quad n\gg 0
\]
Therefore, $B=k[\vec{h}]$ and $\vec{h}$ is a complete infinite signature sequence.
\end{proof}

Let $\vec{h}$ be a signature sequence of length $L$ for the pair $(B,\deg )$, with degree sequence $d_i$. Define monoids $H_i\subset G$ by
$H_i= \N d_1+\cdots+\N d_i$, $1\le i\le L$.

\begin{proposition}\label{intersect} Given $n\ge 1$ let $\vec{h}$ be a signature sequence for $(B,\deg )$ of length $L\ge n$. 
\begin{itemize}
\item [{\bf (a)}] $\mathcal{V}_{d_n}\subset k[h_1,\hdots ,h_{n-1}]$
\item [{\bf (b)}] Given $g\in H_{n-1}$, let $\mathcal{W}_g$ be a subspace of $\mathcal{F}_g$ such that $\mathcal{F}_g=\mathcal{V}_g\oplus\mathcal{W}_g$. If $g\le d_n$, then:
\[
\mathcal{W}_g\cap k[h_1,\hdots ,h_{n-1}]\ne \{ 0\}
\]
\end{itemize}
\end{proposition}

\begin{proof} Assume $f\in B\setminus k[h_1,\hdots ,h_{n-1}]$, and set $g =\deg f$. Then:
\[
f\in \mathcal{F}_g\setminus k[h_1,\hdots ,h_{n-1}]\implies g\ge d_n
\]
This proves part (a).

For part (b), since $g\in H_{n-1}$, there exist $c_1,\hdots ,c_{n-1}\in\N$ such that $\deg (h_1^{c_1}\cdots h_{n-1}^{c_{n-1}})=g$. Therefore, there exist $v\in\mathcal{V}_g$ and nonzero $w\in\mathcal{W}_g$ such that $h_1^{c_1}\cdots h_{n-1}^{c_{n-1}}=v+w$. Consequently:
\[
w-h_1^{c_1}\cdots h_{n-1}^{c_{n-1}}=v\in\mathcal{V}_g \implies \deg (w-h_1^{c_1}\cdots h_{n-1}^{c_{n-1}})<g\le d_n
\]
Part (a) implies $w-h_1^{c_1}\cdots h_{n-1}^{c_{n-1}}\in k[h_1,\hdots ,h_{n-1}]$, so $w\in k[h_1,\hdots ,h_{n-1}]$. This proves part (b).
\end{proof}

\begin{corollary}\label{irred} Given $n\ge 1$ let $\vec{h}$ be a signature sequence for $(B,\deg )$ of length $L\ge n$. 
If $\mathcal{F}_0=k$, then $h_n+b$ is irreducible in $B$ for all $b\in\mathcal{V}_{d_n}$. 
\end{corollary}

\begin{proof}
Assume that $h_n+b=uv$ for $u,v\in B$. If $\deg u<d_n$ and $\deg v<d_n$, then by {\it Proposition\,\ref{intersect}(a)} we have:
\[
u,v,b\in k[h_1,\hdots ,h_{n-1}] \implies uv=h_n+b\in k[h_1,\hdots ,h_{n-1}] \implies h_n\in k[h_1,\hdots ,h_{n-1}] 
\]
a contradiction. Therefore, either $\deg u\ge d_n$ or $\deg v\ge d_n$. Assume that $\deg u\ge d_n$. Then:
\[
d_n\le d_n+\deg v\le \deg u +\deg v=\deg (h_n+b)=d_n \implies \deg v=0 \implies v\in k^*
\]
Likewise, $u\in k^*$ if $\deg v\ge d_n$. 
\end{proof}

\subsection{Homogeneous Signature Sequences}

In this section, let $B$ be a $k$-domain with $G$-grading $\mathfrak{g}$ given by $B=\bigoplus_{g\in G}B_g$, where $(G,<)$ is a finitely generated totally ordered abelian group and $\deg_{\mathfrak{g}}(B)\subset G_+$.

If $\vec{h}$ is a signature sequence for $(B,\deg_{\mathfrak{g}})$ and each $h_n$ is homogeneous, 
we say that $\vec{h}$ is a {\bf homogeneous} signature sequence.
Note that, if $B_0=k$ and $\vec{h}$ is a homogeneous signature sequence, {\it Corollary\,\ref{irred}} implies that any $\beta\in B$ with $\bar{\beta}=h_n$ is irreducible and $\beta-h_n\in k[h_1,\hdots ,h_{n-1}]$.

\begin{lemma}\label{ADD} 
If $\vec{h}=\{ h_i\}$ is a signature sequence for $(B,\deg_{\mathfrak{g}})$, then $\{\bar{h}_i\}$ is a homogeneous signature sequence for $(B,\deg_{\mathfrak{g}})$ and, 
for $0\le n\le\vert\vec{h}\vert$, $k[\bar{h}_1,\hdots ,\bar{h}_n]=k[h_1,\hdots ,h_n]$. 
\end{lemma}

\begin{proof} Given $i\ge 0$, let $b_i=h_i-\bar{h}_i$, noting that $\deg b_i<\deg h_i$. By definition, $\bar{h}_0=h_0$ and $h_1-\bar{h}_1\in k$, so $k[\bar{h}_1]=k[h_1]$. Assume, for some $n$ with 
$1\le n<\vert\vec{h}\vert$, that $\{ \bar{h}_0,\hdots ,\bar{h}_n\}$ is a signature sequence and $k[\bar{h}_1,\hdots ,\bar{h}_n]=k[h_1,\hdots ,h_n]$. We have:
\[
h_{n+1}\in\mathcal{F}_{d_{n+1}}\setminus k[\bar{h}_1,\hdots ,\bar{h}_n]
\]
Therefore, $\bar{h}_{n+1}\in B_{d_{n+1}}$. Since $\deg b_{n+1}<d_{n+1}$, {\it Proposition\,\ref{intersect}} implies that $b_{n+1}\in k[\bar{h}_1,\hdots ,\bar{h}_n]$. 
If $\bar{h}_{n+1}\in k[\bar{h}_1,\hdots ,\bar{h}_n]$, then $h_{n+1}=b_{n+1}+\bar{h}_{n+1}\in k[\bar{h}_1,\hdots ,\bar{h}_n]$, a contradiction. So 
$\bar{h}_{n+1}\not\in k[\bar{h}_1,\hdots ,\bar{h}_n]$. Therefore, 
$\{ \bar{h}_0,\hdots ,\bar{h}_{n+1}\}$ is a signature sequence and:
\[
k[\bar{h}_1,\hdots ,\bar{h}_{n+1}]=k[h_1,\hdots ,h_n, h_{n+1}-b_{n+1}]=k[h_1,\hdots ,h_n, h_{n+1}]
\]
The result now follows by induction on $n$.
\end{proof}

The next result gives criteria for a set of generators in a graded $k$-algebra to be a complete homogeneous signature sequence.

\begin{proposition}\label{criteria} Let $B$ be a $k$-domain with $G$-grading $\mathfrak{g}$, where $(G,<)$ is a finitely generated totally ordered abelian group and $\deg_{\mathfrak{g}}(B)\subset G_+$.
Let $\{ x_i\}_{i\in I}\subset B$ be a set of homogeneous elements which generate $B$ as a $k$-algebra, where $x_0=1$. The following conditions are equivalent. 
\begin{enumerate}
\item 
$\deg_{\mathfrak{g}} x_i\le \deg_{\mathfrak{g}}x_{i+1}$ and 
$x_{i+1}\not\in k[x_1,\hdots ,x_i]$ when $i,i+1\in I\setminus\{ 0\}$.
\smallskip
\item $\vec{x}:=\{ x_i\}_{i\in I}$ is a complete homogeneous signature sequence for $(B,\deg_{\mathfrak{g}})$.
\end{enumerate}
\end{proposition}

\begin{proof} That condition (2) implies condition (1) follows from {\it Definition\,\ref{sig-seq}}. 

Assume that condition (1) holds. 
Let $d_i=\deg_{\mathfrak{g}}x_i$ for $i\in I$.  Clearly, $\{ x_0\}$ is a homogeneous signature sequence. Suppose, by way of induction, that $\vec{p}:=\{x_0,\hdots ,x_m\}$ is a homogeneous signature sequence for some $m\ge 0$. If $\vec{p}$ is complete, then there is nothing further to show. Otherwise, $\vec{p}$ can be extended to a signature sequence $\{ x_0,\hdots ,x_m,y\}$ for some $y\in B$. Let $\rho=\deg_{\mathfrak{g}}y$. 
By {\it Lemma\,\ref{ADD}}, we may assume that $y$ is homogeneous.  By definition of signature sequence and the hypothesis that $x_{m+1}\not\in k[x_1,\hdots ,x_m]$, we have $d_m\le \rho \le d_{m+1}$.

Since $B$ is generated by $\{ x_i\}_{i\in I}$, there exists $t\ge m+1$ such that $y\in k[x_1,\hdots ,x_t]\setminus k[x_1,\hdots ,x_{t-1}]$. Write
\[
y=\sum_{e\in\N^t}\alpha_e x_1^{e_1}\cdots x_t^{e_t} \quad (e_i\in\N\,\, ,\,\, e=(e_1,\hdots ,e_t)\,\, ,\,\, \alpha_e\in k)
\]
where each monomial $x_1^{e_1}\cdots x_t^{e_t}$ with $\alpha_e\ne 0$ is of degree $\rho$. Choose $e$ so that $\alpha_e\ne 0$ and $e_t\ge 1$. Then:
\[
\rho =\deg_{\mathfrak{g}}x_1^{e_1}\cdots x_t^{e_t} \implies
d_{m+1}\ge \rho\ge e_td_t\ge e_td_{m+1}
\]
Therefore, $e_t=1$ and $\rho =d_{m+1}$. 
We have thus shown:
\[
d_{m+1}=\min\{ g\in G\, |\, \mathcal{F}_g\not\subset k[x_1,\hdots ,x_m]\} \quad {\rm and} \quad x_{m+1}\in\mathcal{F}_{d_{m+1}}\setminus k[x_1,\hdots ,x_m]
\]
It follows that $\{ x_0,\hdots ,x_m,x_{m+1}\}$ is a homogeneous signature sequence. 
The desired result now follows by induction. 
\end{proof}

\subsection{Signature Sequences in UFDs}\label{UFD-sig}

\begin{proposition}\label{fact-closed} Assume that $B$ is a UFD over $k$ and that $\mathcal{F}_0=k$. Let $\vec{h}=\{ h_i\}_{i\in I}$ be a signature sequence for $(B,\deg )$.
\begin{itemize}
\item [{\bf (a)}]  $h_i$ is a prime element of $B$ for each $i\in I\setminus\{ 0\}$.
\item [{\bf (b)}]  If $k$ is algebraically closed, then $k[h_i]$ is factorially closed in $B$ for each $i\in I\setminus\{ 0\}$.
\item [{\bf (c)}]  If $k$ is algebraically closed and $m,n\in I$ are such that $1\le m<n$, then $k[h_m,h_n]\cong k^{[2]}$. 
\end{itemize}
\end{proposition}

\begin{proof} For part (a), {\it Corollary\,\ref{irred}} shows that $h_i$ is irreducible in $B$. Since $B$ is a UFD, $h_i$ is prime. 

For part (b), suppose that $uv\in k[h_i]$ for $u,v\in B\setminus k$. Since $k$ is algebraically closed and $uv\not\in k$, $uv$ has a divisor of the form $h_i-\lambda$ where $\lambda\in k$. 
Since $\{ h_0,\hdots ,h_{i-1},h_i\}$ is a signature sequence and $i\ge 1$, $\{ h_0,\hdots ,h_{i-1},h_i-\lambda \}$ is also a signature sequence. 
By part (a), $h_i-\lambda$ is prime in $B$. 
It follows that every prime factor of $u$ (respectively, $v$) is of the form $h_i-\lambda$ for some $\lambda\in k$. Therefore, $u\in k[h_i]$ (respectively, $v\in k[h_i]$). 

For part (c), part (a) implies $k[h_m]$ is factorially closed in $B$, hence algebraically closed in $B$. Since $h_n\not\in k[h_m]$, it follows that $h_n$ is transcendental over $k[h_m]$. Since $m>0$, $k[h_m]\cong k^{[1]}$. Therefore, $k[h_m,h_n]\cong k^{[2]}$.
\end{proof}

%%%%%%%%%%%%%%%%%%%%%%%%%%%%%%%%%%%%%%%%%%%%%%%%%%%%%%%%%%%%%%%%

\section{The rings $k[\Delta ]$}\label{R-Delta}

Let $k$ be a field and let trinomial data $\Delta$ over $k$ be given as in the {\it Introduction}. 
\begin{itemize}
\item [($\Delta$.1)] An integer $n\ge 2$ and partition $n=n_0+\cdots +n_r$ where $r,n_i\ge 1$, $0\le i\le r$.
\item [($\Delta$.2)] A sequence $\beta_i=(\beta_{i1},\hdots ,\beta_{in_i})\in\Z_+^{n_i}$, $0\le i\le r$, where 
$\gcd (d_i,d_j)=1$ when $i\ne j$, $d_i=\gcd(\beta_{i1},\hdots ,\beta_{in_i})$. 
\item [($\Delta$.3)] A sequence of distinct elements $\lambda_2,\hdots ,\lambda_r\in k^*$.
\end{itemize}
As before:
\[
k[\Delta ]=k[T_0,\hdots ,T_r]/(T_0^{\beta_0}+\lambda_iT_1^{\beta_1}+T_i^{\beta_i})_{2\le i\le r}=k[t_{ij}\, |\, 0\le i\le r\, ,\, 1\le j\le n_i ]
\]
Let $\mathfrak{m}\subset k[\Delta ]$ be the maximal ideal generated by the images of the $t_{ij}$. 
The point $p\in X={\rm Spec}(k[\Delta ])$ defined by $\mathfrak{m}$ is called the $\Delta$-{\bf origin} of $X$. 
\begin{definition}{\rm 
\begin{enumerate}
\item Given $i\in\{ 0,\hdots ,r\}$, $|\beta_i|=\sum_{1\le j\le n_i}\beta_{ij}$. 
\item $\Delta$ is {\bf reduced} if either \subitem (i) $r=1$ and $\beta_{01}=\beta_{11}=1$, or \subitem (ii) $r\ge 2$ and $|\beta_i|\ge 2$ for each $i\in\{ 0,\hdots ,r\}$.
\end{enumerate}
}
\end{definition}

\begin{proposition}\label{tri-data}
Let $k$ be a field and $B=k[\Delta ]$ for trinomial data $\Delta$ over $k$.
\begin{itemize}
\item [{\bf (a)}] $B$ admits an effective unmixed $\Z^{d-1}$-grading $\mathfrak{g}$ such that $B_0=k$, where each $t_{ij}$ is homogeneous. 
\item [{\bf (b)}] $B^*=k^*$
\item [{\bf (c)}] If $k$ is algebraically closed, then there exists reduced trinomial data $\Delta^{\prime}$ such that $B\cong_kk[\Delta^{\prime}]$. 
\end{itemize}
\end{proposition}

\begin{proof}
Part (a) follows from \cite{Hausen.Herppich.Suss.11}, Construction 1.7 and Lemma 1.8. The authors work over an algebraically closed ground field of characteristic zero, 
but the grading they construct is valid over any field.

Part (b) is implied by part (a) and {\it Proposition\,\ref{degree}(e)}.

For part (c), assume that $k$ is algebraically closed and that $r\ge 2$ and $|\beta_i|=1$ for some $i$. Then $n_i=1$, $\beta_{i1}=1$ and $T_i^{\beta_i}=t_{i1}$. 
If $i\ge 2$ then the generator $t_{i1}$ and the relation $T_0^{\beta_0}+\lambda_iT_1^{\beta_1}+t_{i1}=0$ can be eliminated, thus reducing both $n$ and $r$ by one. 
$\Delta^{\prime}$ is gotten by eliminating $n_i$, $\beta_i$ and $\lambda_i$ from $\Delta$, and re-indexing the data. 

If $i=0$ then the relations $t_{01}+\lambda_jT_1^{\beta_1}+T_j^{\beta_j}=0$, $j\ge 3$, become:
\[
-T_2^{\beta_2}+(\lambda_j-\lambda_2)T_1^{\beta_1}+T_j^{\beta_j}=0 \quad (3\le j\le r)
\]
The generator $t_{01}$ can be eliminated, thus reducing both $n$ and $r$ by one. 
$\Delta^{\prime}$ is gotten from $\Delta$ by eliminating $n_0$ and $\beta_0$, replacing $\lambda_j$ by $\lambda_j-\lambda_2$ when $j\ge 3$ and $t_{21}$ by $(-1)^{1/\beta_{21}}t_{21}$, and re-indexing the data. 

If $i=1$ then the relations $T_0^{\beta_0}+\lambda_jt_{11}+T_j^{\beta_j}=0$, $j\ge 3$, become:
\[
T_0^{\beta_0}+\lambda_j(\lambda_j-\lambda_2)^{-1}T_2^{\beta_2}+\lambda_2(\lambda_2-\lambda_j)^{-1}T_j^{\beta_j}=0\quad (3\le j\le r)
\]
The generator $t_{11}$ can be eliminated, thus reducing both $n$ and $r$ by one. 
Let $\tilde{\Delta}$ be gotten from $\Delta$ by eliminating $n_1$ and $\beta_1$, replacing $\lambda_j$ by $\lambda_j(\lambda_j-\lambda_2)^{-1}$ when $j\ge 3$, 
and re-indexing the data. Then $B\cong_kk[\tilde{\Delta}]$. The isomorphism of $k[\tilde{\Delta}] $ with $k[\Delta^{\prime}]$ is gotten by 
replacing $t_{j1}$ by 
\[
(\lambda_2(\lambda_2-\lambda_j)^{-1})^{1/\beta_{j1}}t_{j1} 
\]
for $3\le j\le r$. 
\end{proof} 

The grading $\mathfrak{g}$ of $B$ described in the proof of part (a) above is the {\bf grading of $B$ induced by $\Delta$}. 
The fixed points of the torus action on $X$ induced by this grading are determined by $B_0$ and, since $B_0=k$, we see that 
the $\Delta$-origin $p\in X$ is the unique fixed point of this torus action.

\begin{theorem}\label{generators} Let $B=k[\Delta ]$ for reduced trinomial data $\Delta$ using a partition of the integer $n\ge 2$.
Let $X={\rm Spec}(B)$ and let $p\in X$ be the $\Delta$-origin of $X$. 
The dimension of the tangent space to $X$ at $p$ equals $n$. Consequently, 
the minimum number of generators of $B$ as a $k$-algebra is $n$. 
\end{theorem}
\begin{proof} By definition of reduced, there are two cases to consider.

Assume that $r=1$ and $\beta_{01}=\beta_{11}=1$. Then $X$ is an affine space of dimension $n$ and the theorem holds in this case. 

Assume that $r\ge 2$ and $|\beta_i|\ge 2$ for each $i\in\{ 0,\hdots ,r\}$.
Let $m$ be the minimum number of generators of $B$ over $k$. Then clearly $m\le n$. 
Given $i$ with $2 \leq i \leq r$, let $f_i = T_0^{\beta_0} + \lambda_iT_1^{\beta_1} + T_i^{\beta_i}$ in the polynomial ring $k[T_0,\hdots ,T_r]=k[t_{ij} \ | \ 0\leq i\leq r, 1\leq j\leq n_i]$. Let $J$ be the Jacobian matrix of $(f_2,\hdots ,f_r)$, namely:
\[
J=\left( \frac{\partial f_{\ell}}{\partial t_{ij}}\right)_{2\le \ell\le r\, ,\,  0\leq i\leq r\, ,\,  1\leq j\leq n_i}
\]
Then $J$ is of dimension $(r-1)\times n$. For a closed point $q \in X$, we denote by $J(q)$ the Jacobian matrix at $q$, that is:
\[
J(q)=\left( \frac{\partial f_{\ell}}{\partial t_{ij}}(q)\right)_{2\le \ell\le r\, ,\,  0\leq i\leq r\, ,\,  1\leq j\leq n_i}
\]
Let $\mathfrak{m}$ be the maximal ideal of $p\in X$. Since $|\beta_i|\ge 2$ for each $0\leq i\leq r$, we see that ${\rm rank} (J(p)) = 0$, and we have: 
	\begin{center}
	  $\dim_k (\mathfrak{m} / {\mathfrak{m}}^2) = n - {\rm rank} (J(p)) = n$
	\end{center}
Therefore, the dimension of the tangent space at $p$ is $n$, which implies $m \geq n$. 
\end{proof}

\begin{corollary} \label{smooth}Assume that $k$ is algebraically closed. Given trinomial data $\Delta$ over $k$, if $X={\rm Spec}(k[\Delta ])$ is smooth then $X\cong_k\A_k^n$. 
\end{corollary}

\begin{proof} Assume that $X$ is smooth and let $d$ be the dimension of $X$ as a $k$-variety. 
By {\it Proposition\,\ref{tri-data}}, we may assume that $\Delta$ is reduced. 
Then {\it Theorem\,\ref{generators}} implies that $d=n$. From {\it Theorem\,\ref{DFN}} we 
have $d=n-r+1$. Combining these gives $r=1$, so $k[\Delta ]\cong_kk^{[n]}$. 
\end{proof}

See also \cite{Kambayashi.Russell.82}, Theorem 2.5. 

\begin{corollary}\label{tri-data2} Let $k$ be a field and $B=k[\Delta ]$ for reduced trinomial data $\Delta$ over $k$.
Set $d=\dim B$ and let $\bar{t}_{ij}$ be the image of the variable $t_{ij}$ in $k[\Delta ]$. 
There is an ordering $\{ h_m\}_{1\le m\le n}$ of the set $\{ \bar{t}_{ij}\, |\, 0\le i\le r\, ,\, 1\le j\le n_i\}$ 
making $\vec{h}:=\{ h_m\}_{0\le m\le n}$ a complete homogeneous signature sequence for $(B,\deg_{\mathfrak{g}})$, where $h_0=1$.
\end{corollary}

\begin{proof}
Since $\deg_{\mathfrak{g}}(B)$ is a totally ordered set, there is an ordering $\{ h_m\}_{1\le m\le n}$ of the set 
\[
\{ \bar{t}_{ij}\, |\, 0\le i\le r\, ,\, 1\le j\le n_i\}
\]
 such that $\deg_{\mathfrak{g}} h_1\leq\cdots\leq\deg_{\mathfrak{g}} h_n$. By {\it Theorem\,\ref{generators}}, 
\[
h_{i+1}\not\in k[h_1,\ldots,h_{i},h_{i+2},\ldots,h_n]
\]
and, in particular, $h_{i+1}\not\in k[h_1,\ldots,h_{i}]$. {\it Proposition\,\ref{criteria}} implies that $\vec{h}:=\{ h_m\}_{0\le m\le n}$ is a complete homogeneous signature sequence for $(B,\deg_{\mathfrak{g}})$. 
\end{proof}

%%%%%%%%%%%%%%%%%%%%%%%%%%%%%%%%%%%%%%%%%%%%%%%%%%%%%%%%%%%%%%%%

\section{Graded Unirational UFDs}\label{Dim-n}

\begin{theorem}\label{classify}  Assume that the field $k$ is algebraically closed. Let $B$ be a unirational UFD of finite transcendence degree $n\ge 1$ over $k$.
If $B$ admits an effective unmixed $\Z^{n-1}$-grading with $B_0=k$, then $B\cong_kk[\Delta ]^{[m]}$ for some reduced trinomial data $\Delta$ and $m\in\N$. 
\end{theorem}

\begin{proof} By {\it Theorem\,\ref{Knop}}, $B$ is affine over $k$, and by {\it Corollary\,\ref{affineUFD}}, $B$ is rational over $k$. 
Let $\mathfrak{g}$ be the given grading of $B$ with decomposition $B=\oplus_{g\in G}B_g$ and set $\deg =\deg_{\mathfrak{g}}$. By hypothesis, the degree monoid of $M_{\mathfrak{g}}(B)$ is unmixed, 
and by {\it Proposition\,\ref{unmixed-monoid}}, there exists a total order $<$ on $\Z^{n-1}$ making $M_{\mathfrak{g}}(B)$ a well-ordered set. 
The remainder of the proof consists of showing the following three statements. 
\begin{enumerate}
\item 
There exist ${\bf x}=(x_1,\hdots ,x_l)\in B^l$, ${\bf y}=(y_1,\hdots ,y_p)\in B^p$, ${\bf z}=(z_1,\hdots ,z_q)\in B^q$ and ${\bf w}=(w_1,\hdots ,w_N)\in B^N$ with $l,p,q,N\in\N$, $l+p+q=n$ and $p,q\ge 1$, such that
\[
B=k[{\bf x},{\bf y},{\bf z},{\bf w}] \quad {\rm and}\quad {\bf x}^{{\bf r}_i}w_i^{e_i}+a_i{\bf y}^{\bf s}+b_i{\bf z}^{\bf t}=0 \,\, (1\le i\le N)
\]
where ${\bf r}_i\in\N^l$, ${\bf s}\in\N^p$, ${\bf t}\in\N^q$, $e_i\in\Z_+$ for $1\le i\le N$, and $(a_1,b_1),\hdots ,(a_N,b_N)\in (k^*)^2$ are pairwise linearly independent. 
\item The following properties hold:
\subitem (a) the components of ${\bf x},{\bf y},{\bf z}$ and ${\bf w}$ form a (finite) complete homogeneous signature sequence $\vec{h}$ for $(B,\deg )$ in some order, and $N=|\vec{h}|-n$;
\subitem (b) the components of ${\bf x},{\bf y}$ and ${\bf z}$ form a set of $n$ algebraically independent elements of $\vec{h}$ whose degrees generate a group of rank $n-1$;
\subitem (c) $\tau:=\deg {\bf x}^{{\bf r}_i}w_i^{e_i}=\deg {\bf y}^{\bf s}=\deg {\bf z}^{\bf t}$ for each $i$; 
\subitem (d) if $K_0=\{ u/v \, |\, u,v\in B_g\, ,\, v\ne 0\, ,\, g\in G\}$ then $K_0=k({\bf y}^{\bf s}{\bf z}^{-{\bf t}})$;
\subitem (e) If $u$ and $v$ are $k$-linearly independent monomials in elements of $\vec{h}$ of degree $\tau$, then $K_0=k(u/v)$ and $\gcd_B(u,v)=1$.
\item $B\cong_kk[\Delta ]^{[m]}$ for some $m\in\N$ and trinomial data $\Delta$. 
\end{enumerate}

Let $K_0=\{ u/v \, |\, u,v\in B_g\, ,\, v\ne 0\, ,\, g\in G\}$. By {\it Corollary\, \ref{unirational}}, $K_0=k(\zeta )=k^{(1)}$ for some $\zeta\in K_0$.
Let $\zeta =u/v$ for $u,v\in B_{\delta}$ for $\delta\in G$ and $\gcd_B(u,v)=1$. 
If $K_0=k(u'/v')$ for $u',v'\in B_g$ with $\gcd_B(u',v')=1$, then $g=\delta$ and there exists $A\in GL_2(k)$ such that $(u',v')=(u,v)A$. In order to see this, write:
\[
\frac{u'}{v'}=\frac{au+bv}{cu+dv} \quad {\rm for} \quad \begin{pmatrix} a&b\cr c&d\end{pmatrix}\in GL_2(k)
\]
Then $u'(cu+dv)=v'(au+bv)$. Since $\gcd_B(u',v')=1$, $u'$ divides $au+bv$ and $v'$ divides $cu+dv$. Therefore, $g\le\delta$, and by symmetry, $\delta\le g$, so $g=\delta$. It follows that
$u'=\lambda (au+bv)$ and $v'=\mu(cu+dv)$ for some $\lambda ,\mu\in k^*$. 

If $B\cong k^{[n]}$, then $B=k[\Delta ]$ for trinomial data with $r=1$. Assume hereafter that $B\not\cong k^{[n]}$. 
Since $B$ is affine, $n=\dim B$ and $B\not\cong k^{[n]}$, there exists 
$N_0\geq n+1$ and $\vec{h}=\{h_0,h_1,\ldots,h_{N_0}\}\subset B$ which are homogeneous such that $h_0=1$ and $\{h_1,\ldots,h_{N_0}\}$ is a minimal system of generators of $B$, that is, $B=k[h_1,\ldots,h_N]$ and $h_i\not\in k[h_1,\ldots,h_{i-1},h_{i+1},\ldots,h_{N_0}]$ for any $i$. By reordering, we may assume that $\deg h_i\leq\deg h_{i+1}$ for $i\in\{1,\ldots,N_0-1\}$. {\it Proposition\, \ref{criteria}} implies that $\vec{h}$ is a complete homogeneous signature sequence for $(B,\deg)$. 
By {\it Proposition\, \ref{fact-closed}}, the $h_i$ are pairwise relatively prime in $B$. 
Let $\{ d_i\}\subset G$ be the degree sequence of $\vec{h}$, $0\le i\le\vert\vec{h}\vert$.  
We call ${\bf h}^{\bf r}\in B$ a {\bf monomial} if it is of the form ${\bf h}^{\bf r}=h_{i_1}^{r_1}\cdots h_{i_s}^{r_s}$, where $(h_{i_1},\hdots ,h_{i_s})$ is a finite subsequence of 
$\vec{h}$ and ${\bf r}=(r_1,\hdots ,r_s)\in\N^s$. 

\medskip

\noindent {\bf Claim 1.} {\it Assume that ${\bf h}_1^{{\bf r}_1},{\bf h}_2^{{\bf r}_2}\in B$ are relatively prime monomials such that ${\bf h}_1^{{\bf r}_1}{\bf h}_2^{-{\bf r}_2}\in K_0$, where ${\bf r}_1\in\N^{s_1}$
and ${\bf r}_2\in\N^{s_2}$. There exist $\ell\in\N$, $\ell\ge 1$, and ${\bf t}_j\in\N^{s_1}$, ${\bf w}_j\in\N^{s_2}$, $1\le j\le\ell$, such that $\sum_j{\bf t}_j={\bf r}_1$, $\sum_j{\bf w}_j={\bf r}_2$, and 
$K_0=k\left( {\bf h}_1^{{\bf t}_i}{\bf h}_2^{-{\bf w}_j}\right)$ for each pair $i,j$.}

\medskip

\noindent {\it Proof of Claim 1.} Let $k[X,Y]\cong k^{[2]}$. By hypothesis, there exist standard homogeneous $F,G\in k[X,Y]=k^{[2]}$ of the same degree $\ell$ such that:
\[
{\textstyle\gcd_{k[X,Y]}(F(X,Y),G(X,Y))}=1 \quad {\rm and}\quad {\bf h}_1^{{\bf r}_1}\,G(u,v)={\bf h}_2^{{\bf r}_2}F(u,v)
\]
Let $L_i,M_j\in k[X,Y]$, $1\le i,j\le\ell$, be linear forms such that $F=L_1\cdots L_{\ell}$ and $G=M_1\cdots M_{\ell}$. 
If $\lambda\in B$ is prime and $F(u,v),G(u,v)\in\lambda B$, then $L_i(u,v),M_j(u,v)\in\lambda B$ for some $i,j$. 
Since $L_i$ and $M_j$ are linearly independent, $u,v\in\lambda B$, which implies $\lambda\in k^*$. 
Therefore, $\gcd_B(F(u,v),G(u,v))=1$. From the equality ${\bf h}_1^{{\bf r}_1}\,G(u,v)={\bf h}_2^{{\bf r}_2}F(u,v)$ it follows that, for each $j$, $1\le j\le\ell$, 
\[
L_j(u,v)={\bf h}_1^{{\bf t}_j} \quad {\rm and}\quad M_j(u,v)={\bf h}_2^{{\bf w}_j}
\]
where ${\bf t}_j\in\N^{s_1}$ with $\sum_j{\bf t}_j={\bf r}_1$, and ${\bf w}_j\in\N^{s_2}$ with $\sum_j{\bf w_j}={\bf r}_2$.
Since $(L_i,M_j)\in GL_2(k)$ for each pair $i,j$, it follows that:
\[
K_0=k(u/v)=k\left( L_i(u,v)/M_j(u,v)\right) =k\left( {\bf h}_1^{{\bf t}_i}{\bf h}_2^{-{\bf w}_j}\right)
\]
This proves {\it Claim\,1}. 

\medskip

\noindent {\bf Claim 2.} 
{\it For $m\in\Z$ with $m\geq n$, let $\{i_1,\ldots,i_m\}$ be a subset of $\Z$ of cardinality $m$, where $1\le i_j\le N_0$ for each $j$. There exist $\epsilon_1,\hdots ,\epsilon_m\in\Z$ such that $K_0=k(h_{i_1}^{\epsilon_1}\cdots h_{i_m}^{\epsilon_m})$. 
In particular, if $fg^{-1}=h_{i_1}^{\epsilon_1}\cdots h_{i_m}^{\epsilon_m}$ for $f,g\in B$ with $\gcd_B(f,g)=1$, then $f,g\in k[h_{i_1},\ldots, h_{i_m}]\cap B_{\delta}$. }
\medskip

\noindent {\it Proof of Claim 2.} 
Define the mapping $\phi :\Z^m\to G$, $\phi (r_1,\hdots ,r_m)=r_1d_{i_1}+\cdots +r_md_{i_m}$. Since $\phi$ has a nonzero kernel element, {\it Claim\,2} is implied by {\it Claim\,1}. 
\medskip

In order to prove parts (1) and (2), let $E=\{ i_1,\hdots ,i_n\}$ be a subset of $\{ 1,\hdots ,N_0\}$ such that $h_{i_1},\ldots,h_{i_n}$ are algebraically independent over $k$. 
Let $r={\rm rank}\langle d_{i_1},\hdots ,d_{i_n}\rangle$ and let $A=k[h_{i_1},\ldots,h_{i_n}]$. 
Since $A$ is a graded subalgebra of $B$ and the rank of its grading is $r$, {\it Proposition\,\ref{Laurent}} shows that ${\rm frac}(A)\cong L_0^{(r)}$, where $L_0={\rm frac}(A)\cap K_0$. {\it Claim\;2} implies that $L_0\cong k^{(1)}$. 
Therefore, 
	\[
	k^{(n)}\cong{\rm frac}(A)\cong L_0^{(r)}\cong k^{(r+1)}
	\]
which implies that $r=n-1$. 

By {\it Claim\,2}, there are $p,s\in\N$ with $1\le p<s\le n$, $(\epsilon_1,\hdots ,\epsilon_s)\in\Z_+^s$ and a permutation $\sigma$ of $\{i_1,\hdots ,i_n\}$ such that $K_0=k(h_{\sigma (i_1)}^{\epsilon_1}\cdots h_{\sigma (i_p)}^{\epsilon_p}h_{\sigma (i_{p+1})}^{-\epsilon_{p+1}}\cdots h_{\sigma (i_s)}^{-\epsilon_s})$. 
Set $q=s-p>0$, $\ell=n-s\geq 0$ and 
\begin{align*}
{\bf y}&=(y_1,\ldots,y_p)=(h_{\sigma(i_1)},\ldots,h_{\sigma(i_p)}), \\
{\bf z}&=(z_1,\ldots,z_q)=(h_{\sigma(i_{p+1})},\ldots,h_{\sigma(i_{s})}), \\
{\bf x}&=(x_1,\ldots,x_{\ell})=(h_{\sigma(i_{s+1})},\ldots,h_{\sigma(i_n)}), \\
{\bf s}&=(s_1,\ldots,s_{p})=(\epsilon_1,\hdots ,\epsilon_p),\\ 
{\bf t}&=(t_1,\ldots,t_{q})=(\epsilon_{p+1},\hdots ,\epsilon_{s}). 
\end{align*} 
Then $K_0=k({\bf y}^{\bf s}{\bf z}^{-\bf t})$, $\ell+p+q=(n-s)+p+(s-p)=n$ and the components of ${\bf x},{\bf y}$ and ${\bf z}$ form a set of $n$ algebraically independent elements of $\vec{h}$ whose degrees generate a group of rank $n-1$. 

Let $L\in \{ 1,\hdots ,N_0\}\setminus E$ be given. Since ${\rm rank}\langle d_{i_1},\hdots ,d_{i_n}\rangle =n-1$, there exists $e_L\in\Z\setminus\{0\}$ such that $e_Ld_L\in\langle d_{i_1},\hdots ,d_{i_n}\rangle$.  By the same proof used for {\it Claim\,2}, there exist ${\bf r}_L\in\Z^{\ell}$, ${\bf s}_L\in\Z^p$ and ${\bf t}_L\in\Z^q$ such that $k({\bf x}^{{\bf r}_L}{\bf y}^{{\bf s}_L}{\bf z}^{{\bf t}_L}w_L^{e_L})= K_0$, where $w_L=h_L$. Let $f$ (resp.\:$g$) be the numerator (resp.\:denominator) of ${\bf x}^{{\bf r}_L}{\bf y}^{{\bf s}_L}{\bf z}^{{\bf t}_L}w_L^{e_L}$. Since $k(f/g)=K_0=k({\bf y}^{\bf s}{\bf z}^{-\bf t})$, there exists $A\in GL_2(k)$ such that 
\[
(f,g)=({\bf y}^{\bf s},{\bf z}^{\bf t})A. 
\]
Then there exist $a_L,b_L\in k$ such that $f=a_L{\bf y}^{\bf s}+b_L{\bf z}^{\bf t}$. Without loss of generality, we may assume that $e_L>0$. Then $f$ is divisible by $w_L^{e_L}$. Since the components of ${\bf x},{\bf y},{\bf z}$ and $w_L$ are pairwise relatively prime, $f={\bf x}^{{\bf r}_L}w_L^{e_L}$ and $a_L\not=0$ and $b_L\not=0$. Therefore, ${\bf x}^{{\bf r}_L}w_L^{e_L}=a_L{\bf y}^{\bf s}+b_L{\bf z}^{\bf t}$.

Let $N=N_0-n$. By using the same argument, we get $w_1,\ldots,w_N\in B$ which are elements of $\vec{h}$, $(e_1,\ldots,e_N)\in\Z_+^N$ and pairwise linearly independent elements $(a_1,b_1),\hdots ,(a_N,b_N)\in (k^*)^2$ such that 
\[
{\bf x}^{{\bf r}_i}w_i^{e_i}+a_i{\bf y}^{\bf s}+b_i{\bf z}^{\bf t}=0 \quad (1\leq i\leq N).
\]
Let $M_1,M_2$ be $k$-linearly independent monomials in $h_1,\hdots ,h_N$ of degree $\delta$. Then $M_1/M_2\in K_0$, and by {\it Claim\,1}, there exist relatively prime monomials $\mu_i$ dividing $M_i$ ($i=1,2$) such that $K_0=k(\mu_1/\mu_2)$. As observed, we then have $\deg\mu_i=\delta$ for each $i$, and it follows that $\mu_i=M_i$. 

This completes the proof of parts (1) and (2).

It remains to show part (3). 
Given $i\ne j$, if ${\bf x}^{{\bf r}_i}$ and ${\bf x}^{{\bf r}_j}$ are not relatively prime, then 
$h_{\kappa}$ divides both ${\bf x}^{{\bf r}_i}$ and ${\bf x}^{{\bf r}_j}$ for some $\sigma (i_{s+1})\le\kappa\le\sigma (i_n)$.
Since $h_{\kappa}$ is a prime element of $B$, $h_{\kappa}B$ is a prime ideal containing both 
$a_i{\bf y}^s+b_i{\bf z}^t$ and $a_j{\bf y}^s+b_j{\bf z}^t$. Since $(a_i,b_i)$ and $(a_j,b_j)$ are linearly independent, we have 
${\bf y}^s,{\bf z}^t\in h_{\kappa}B$, a contradiction, since the prime factors of {\bf y} and {\bf z} are $h_{\sigma (i_1)},\hdots ,h_{\sigma (i_s)}$.
Therefore, $\gcd ({\bf x}^{{\bf r}_i},{\bf x}^{{\bf r}_j})=1$ for all $i\ne j$.

Let $J\subset\{ 1,\hdots ,l\}$ be the subset of those $j$ such that $r_{ij}\ne 0$ for some $i$ $(1\le i\le N)$, and let $m=|J|$. By reordering, we may assume that $J=\{ 1,\hdots , m\}$.
Let $\tilde{\bf x}=(x_1,\hdots ,x_m)$ and $S=k[\tilde{\bf x},{\bf y},{\bf z},{\bf w}]$.
Then $B\cong S^{[l-m]}$.
Let us now rewrite 
\[
\tilde{\bf x}^{{\bf r}_i}w_i^{e_i}+a_i{\bf y}^{\bf s}+b_i{\bf z}^{\bf t} = a_i\left( T_0^{\beta_0}+\lambda_iT_1^{\beta_1}+T_i^{\beta_i}\right) \,\, ,\,\, 1\le i\le N
 \]
 where $T_0^{\beta_0}={\bf y}^{\bf s}$, $T_1^{\beta_1}={\bf z}^{\bf t}$, $T_i^{\beta_i}=\tilde{\bf x}^{{\bf r}_i}(\alpha w_i)^{e_i}$, $\lambda_i=b_i/a_i$ and $\alpha\in k$ satisfies $\alpha^{e_i}=a_i$. 
Note that $\lambda_1,\hdots ,\lambda_N\in k^*$ are distinct. 
Let $d_i=\gcd (\beta_{i1},\hdots ,\beta_{in_i})$ and let $D$ be a common divisor of $d_i$ and $d_j$ for $i\ne j$. Then $K_0=k(\zeta^D)$ for $\zeta$ of the form 
${\bf h}_1^{\bf a}{\bf h}_2^{-{\bf b}}$, which implies $D=1$. 
It follows that $S=k[\Delta ]$ for trinomial data $\Delta$. Since $\{\tilde{\bf x},{\bf y},{\bf z},{\bf w}\}$ is a subset of a minimal system of generators $\{h_1,\ldots,h_{N_0}\}$ of $B$, $\Delta$ is reduced. 
This completes the proof of (3). 
\end{proof}

%%%%%%%%%%%%%%%%%%%%%%%%%%%%%%%%%%%%%%%%%%%%%%%%%%%%%%%%%%%%%%%%

\section{Graded UFDs of Dimension Two or Three}\label{Dim-2} 
Throughout this section, it is assumed that the ground field $k$ is algebraically closed. 

\subsection{Dimension Two}

Let $\Delta$ be trinomial data using the integer $n\ge 2$ and partition:
\[
n=n_0+\cdots +n_r\quad  (r,n_i\ge 2)
\]
Then $\dim k[\Delta ]=n-r+1=2$ if and only if $n_i=1$ for each $i$. This is the {\bf unit partition} of $n$. 
Since $\beta_i\in\Z_+^{n_i}$ for each $i$, we see that $\beta_0,\hdots ,\beta_r$ are pairwise relatively prime positive integers. We thus have
\[
k[\Delta ]=k[T_0,\hdots ,T_r]/(T_0^{\beta_0}+\lambda_iT_1^{\beta_1}+T_i^{\beta_i})_{2\le i\le r}
\]
where $\lambda_2,\hdots ,\lambda_r\in k^*$ are distinct. 
The $\Z$-grading of $k[\Delta ]$ induced by $\Delta$ is defined by 
letting the image of $T_i$ be homogeneous of degree $\beta/\beta_i$ where $\beta=\prod_{0\le i\le r}\beta_i$. 

We say that $\Delta$ is in {\bf Mori form} if 
either (1) $r=1$ and $\beta_0=\beta_1=1$ or (2) $r\ge 2$,  $\lambda_2=1$ and $\beta_0>\beta_1>\cdots >\beta_r\ge 2$. 
Note that this is stronger than the assumption that $\Delta$ is reduced. 

\begin{lemma}\label{Mori-form} There exists trinomial data $\Delta^{\prime}$ in Mori form such that $k[\Delta ]\cong_kk[\Delta^{\prime}]$.
\end{lemma}

\begin{proof}
Assume that $r\ge 2$ and 
choose $M$ so that $\beta_M=\max_i\beta_i$. If $M\ge 2$, replace $T_0^{\beta_0}$ by
$-(\lambda_MT_1^{\beta_1}+T_M^{\beta_M})$ to obtain relations
\[
T_M^{\beta_M}+(\lambda_M-\lambda_i)T_1^{\beta_1}+(\xi_iT_i)^{\beta_i}=0 \quad (i\ne M\, ,\, \xi_i\in k\, ,\, \xi_i^{\beta_i}=-1) 
\]
and:
\[
T_M^{\beta_M}+\lambda_MT_1^{\beta_1}+T_0^{\beta_0}=0
\]
If $M=1$ we have relations:
\[
T_1^{\beta_1}+\lambda_i^{-1}T_0^{\beta_0}+(\xi_iT_i)^{\beta_i}=0 \quad (2\le i\le r\, ,\, \xi_i\in k\, ,\, \xi_i^{\beta_i}=\lambda_i^{-1})
\]
Therefore, we may assume, without loss of generality, that $\beta_0\ge \beta_i$ for $i\ge 1$. 

Choose $N$ so that $\beta_N=\max_{i\ge 1}\beta_i$. If $N\ge 2$, replace $T_1^{\beta_1}$ by $-\lambda_N^{-1}(T_0^{\beta_0}+T_N^{\beta_N})$ to obtain relations
\[
T_0^{\beta_0}+\mu_iT_N^{\beta_N}+(\xi_i T_i)^{\beta_i}=0 \quad (i\ne N\, ,\, \mu_i=\lambda_i(\lambda_i-\lambda_N)^{-1}\, ,\, \xi_i\in k\, ,\, \xi_i^{\beta_i}=\lambda_N(\lambda_N-\lambda_i)^{-1})
\]
and:
\[
T_0^{\beta_0}+T_N^{\beta_N}+(\xi_N T_1)^{\beta_1}=0 \quad (\xi\in k\, ,\, \xi^{\beta_1}=\lambda_N)
\]
Therefore, we may assume, without loss of generality, that $\beta_0\ge \beta_1\ge \beta_i$ for $i\ge 2$. 
The remaining monomials $T_i^{\beta_i}$, $2\le i\le r$, can be permuted to obtain $\beta_0\ge \beta_1\ge \beta_2\ge \cdots \ge \beta_r$. 
If $\beta_i=\beta_{i+1}$ for any $i$, then since $\gcd (\beta_i,\beta_{i+1})=1$ we have $\beta_i=\beta_{i+1}=1$. This implies
$\beta_r=1$ and, since $r\ge 2$, $T_r$ can be eliminated and $r$ can be reduced by one. Therefore, we may assume $\beta_0>\beta_1>\beta_2>\cdots >\beta_r\ge 2$ when $r\ge 2$. 
Finally, to get $\lambda_2=1$, replace $T_1$ by $\lambda_2^{1/\beta_1}T_1$ and replace $\lambda_i$ by $\lambda_i/\lambda_2$. 
\end{proof}

For the next result, define trinomial data $\Delta_0$ by the unit partition $3=1+1+1$; $\beta_0=5$, $\beta_1=3$, $\beta_2=2$; and $\lambda_2=1$. Then
$k[\Delta_0]=k[T_0,T_1,T_2]/(T_0^5+T_1^3+T_2^2)$. 

\begin{theorem}\label{stable} Assume that the characteristic of $k$ is zero. Let $\Delta$ be trinomial data over $k$ with unit partition. 
Let $B$ be a UFD over $k$ and $A\subseteq B$ a factorially closed subalgebra isomorphic to $k[\Delta ]$.
\begin{itemize}
\item [{\bf (a)}] If $A$ is not isomorphic to $k^{[2]}$ or $k[\Delta_0]$ then $A\subseteq ML(B)$. 
In particular, $k[\Delta ]$ is stably rigid.
\item [{\bf (b)}] $k[\Delta_0]$ is rigid.
\item [{\bf (c)}] If $B\cong_kk^{[m]}$ for some $m\ge 2$ then $A\cong_kk^{[2]}$. 
\end{itemize}
\end{theorem}

\begin{proof} Part (a). By {\it Lemma\,\ref{Mori-form}}, we may assume, with no loss of generality, that $\Delta$ is in Mori form. Since $k[\Delta ]\not\cong k^{[2]}$, we have $r\ge 2$ and we can write
\[
k[\Delta ]=k[T_0,\hdots ,T_r]/(T_0^{\beta_0}+\lambda_iT_1^{\beta_1}+T_i^{\beta_i})_{2\le i\le r}
\]
where $\beta_0>\cdots >\beta_r\ge 2$ are pairwise relatively prime integers and $\lambda_2,\hdots ,\lambda_r\in k^*$ are distinct. 
Let $h_i$ denote the image of $T_i$ in $A$. 
By {\it Corollary\,\ref{tri-data2}}, $\vec{h}:=\{ h_0,\hdots ,h_r\}$ is a complete homogeneous signature sequence consisting of pairwise nonassociate prime elements of $A$. 
Since $A$ is factorially closed in $B$, $\vec{h}$ consists of pairwise nonassociate primes of $B$.

Suppose that $\beta_0^{-1}+\beta_1^{-1}+\beta_i^{-1}\le 1$ for some $i$. By the ABC Theorem (\cite{Freudenburg.17}, Thm.\,2.48), it follows that $k[h_0,h_1,h_i]\subset ML(B)$. 
Since $A$ is algebraic over $k[h_0,h_1,h_i]$ and $ML(B)$ is algebraically closed in $B$, it follows that $A\subseteq ML(B)$. Taking $B=A^{[n]}$ for $n\ge 0$ shows that $A$ is stably rigid. 
Otherwise $\beta_0^{-1}+\beta_1^{-1}+\beta_i^{-1}> 1$ for each $i$. It is easy to check that $\{ \beta_0,\beta_1,\beta_i\}=\{ 5,3,2 \}$. Therefore, $r=2$ and $\Delta = \Delta_0$, a contradiciton. 
This proves part (a). 

Part (b). The rigidity of $k[\Delta_0]$ is given by \cite{Freudenburg.17}, Theorem 9.7. 

Part (c). 
For $2\le i\le r$ we have $h_0^{\beta_0} +\lambda_ih_1^{\beta_1}+h_i^{\beta_i}=0$, where $\beta_0>\beta_1>\beta_i\ge 2$. 
This is an equation in $k^{[m]}$ where each $h_j$ is prime in $A$, and therefore also prime in both $B$ and $k^{[m]}$. By {\it Lemma\,\ref{Iitaka}} below, 
$\beta_i=1$ for $2\le i\le r$. Therefore, $r=1$ and $A\cong k^{[2]}$. 
\end{proof}

The following result is due to Iitaka (which assumes that $k$ is algebraically closed of characteristic zero). 

\begin{lemma}\label{Iitaka} {\rm \cite{Iitaka.79}.} Given the integer $n\ge 1$, let $f,g,h\in k^{[n]}$ be irreducible pairwise relatively prime polynomials such that $f^a+g^b+h^c=0$ for integers $a,b,c$ with 
$a\ge b\ge c\ge 1$. Then $c=1$.
\end{lemma}

\begin{theorem}\label{max-torus} Assume that the characteristic of $k$ is 0.
Let $B$ be an affine UFD over $k$, $\dim B=2$.
If $B$ admits an effective unmixed $\Z$-grading with $B_0=k$, then either $B\cong k^{[2]}$ or ${\rm Aut}_k(B)$ contains a unique maximal torus $T$ and $\dim T=1$.
\end{theorem}

\begin{proof} 
By {\it Theorem\,\ref{main}} and {\it Lemma\,\ref{Mori-form}}, $B$ is isomorphic to $k[\Delta ]$ for some trinomial data $\Delta$ in Mori form. 
By {\it Corollary\,\ref{tri-data}}, $B^*=k^*$. 
Assume that $B$ is not isomorphic to $k^{[2]}$. 
{\it Theorem\,\ref{stable}} implies that $B$ is rigid. According to \cite{Arzhantsev.Gaifullin.17}, Theorem 2.1, the automorphism group of a rigid affine variety contains a unique maximal torus. 
Let $T\subset {\rm Aut}_k(B)$ be the unique maximal torus. 
{\it Proposition\,\ref{Laurent}}, combined with the hypothesis that $B$ has an effective $\Z$-grading, shows that $1\le\dim T\le 2$. 

Suppose that $\dim T=2$. Then ${\rm Spec}(B)$ is a toric surface. However, this gives a contradiction in two ways:
If $A$ is the coordinate ring of an affine toric surface over $k$ with $A^*=k^*$, then (1) $A$ is not rigid, and (2) $A$ is not a UFD unless $A=k^{[2]}$. 
The first statement can be shown using \cite{Liendo.10}, Proposition 4.2. We give a proof of the second statement. 

It is well known that, since $A^*=k^*$, $A$ is of the form $k[V]^G$ 
for some representation of a finite cyclic group $G$ on $V=k^2$. Let $H$ be the subgroup generated by pseudo-reflections, i.e., elements $g\in G$ whose fixed point set is a line. 
Suppose that $A$ is a UFD. The theorem of Nakajima \cite{Nakajima.82}, Theorem 2.11, implies that $G/H$ has no non-trivial one-dimensional representation. 
Since $G/H$ is cyclic, it follows that $G=H$, i.e., $G$ is generated by pseudo-reflections. 
But then the theorem of Shepard and Todd \cite{Shepard.Todd.54} implies $A\cong k^{[2]}$.

Therefore, $\dim T=1$.
\end{proof} 

The next result gives uniqueness of presentation of rings by trinomial data in Mori form. 

\begin{theorem}\label{non-iso} Assume that the characteristic of $k$ is zero and that $\Delta ,\Delta^{\prime}$ are trinomial data over $k$ in Mori form. 
If $k[\Delta ]\cong_kk[\Delta^{\prime}]$ then $\Delta =\Delta^{\prime}$. 
\end{theorem}

\begin{proof} Let
\[
B=k[\Delta ]=k[T_0,\hdots ,T_r]/(T_0^{\beta_0}+\lambda_iT_1^{\beta_1}+T_i^{\beta_i})_{2\le i\le r}
\]
where $\beta_0> \cdots >\beta_r\ge 2$ are pairwise relatively prime positive integers and $\lambda_2,\hdots ,\lambda_r\in k^*$ are distinct. 
Since $r=1$ if and only if $k[\Delta ]\cong k^{[2]}$, it suffices to assume that $r\ge 2$.
Let $\mathfrak{g}$ be the $\Z$-grading of $B$ induced by $\Delta$. By {\it Theorem\,\ref{max-torus}}, the only effective $\Z$-gradings of $B$ are $\pm\mathfrak{g}$. 

Let $h_i$ be the image of $T_i$ in $B$, $0\le i\le r$.
By {\it Corollary\,\ref{tri-data2}}, $\vec{h}:=\{1,h_0,\hdots ,h_r\}$ is a complete homogeneous signature sequence for $(B,\deg_{\mathfrak{g}} )$. 
Let $d_i$ be the degree of $h_i$. Then $d_i=\beta/\beta_i$ for $\beta=\beta_0\cdots \beta_r$. Suppose that, for some $i\ge 1$, 
$d_i=\sum_{j=0}^{i-1}a_jd_j$ for $a_j\in\Z$.
Then
\[
\frac{\beta}{\beta_i}=\sum_{j=0}^{i-1}a_j\frac{\beta}{\beta_j}=\frac{m\beta}{\beta_0\cdots \beta_{i-1}} \implies \beta_0\cdots \beta_{i-1}=m\beta_i
\]
for some positive integer $m$. But this contradicts the hypothesis that $\gcd(\beta_j,\beta_i)=1$ for $j<i$. Therefore, 
for each $i\ge 1$, $d_i$ is not in the subgroup of $\Z$ generated by $d_0,\hdots ,d_{i-1}$. It follows that $B_{d_i}\cap k[h_0,\hdots ,h_{i-1}]=\{ 0\}$ for $i\ge 1$, and that $\{ d_i\}_{0\leq i\leq r}$ is a strictly increasing sequence. Consequently:
\[
B_{d_i}=k\cdot h_i \quad {\rm for}\quad 0\le i\le r
\]

Now suppose that $\Delta^{\prime}$ is defined by pairwise relatively prime integers $\gamma_0> \cdots >\gamma_t\ge 2$ and distinct $\mu_2,\hdots ,\mu_t\in k^*$, where $t\ge 2$. 
Let $\varphi :k^{[t]}=k[S_0,\hdots ,S_t]\to B$ be a $k$-algebra surjection such that
\[
\krn\varphi = (S_0^{\gamma_0}+\mu_jS_1^{\gamma_1}+S_j^{\gamma_j})_{2\le j\le t}
\]
By {\it Theorem\,\ref{generators}} we see that $t=r$, and by our previous observation, the $\Z$-grading of $B$ induced by $\Delta^{\prime}$ must equal $\mathfrak{g}$. 
Let $h_i^{\prime}$ be the image of $S_i$ in $B$. As above, $\vec{h}^{\prime}:=\{1,h_i^{\prime}\}_{0\le i\le r}$ is a complete homogeneous signature sequence for $(B,\mathfrak{g})$. 
We claim that $h_i^{\prime}\in k\cdot h_i$ for each $i$. This is clear for $i=0$, so assume it holds for $0\le j\le i$. Since $h_{i+1}\not\in k[h_0,\hdots ,h_i]=k[h_0^{\prime},\hdots ,h_i^{\prime}]$, we see that 
$\deg h_{i+1}^{\prime}\le d_{i+1}$ by minimality of $\deg h_{i+1}^{\prime}$. Conversely, by minimality of $d_{i+1}$, we see that $d_{i+1}\le \deg h_{i+1}^{\prime}$. Therefore $\deg h_{i+1}^{\prime} = d_{i+1}$ and $h_{i+1}^{\prime}\in B_{d_{i+1}}=k\cdot h_{i+1}$. So the claim follows by induction. 

We now have $d_i=\beta/\beta_i=\gamma/\gamma_i$ for each $i$. So for any subset $\{ a,b,c\}\subset\{ 0,\hdots ,r\}$ of three distinct elements we have: 
\[
\beta_a=\frac{\gamma/\gamma_b}{\beta/(\beta_a\beta_b)}=\frac{\gamma/\gamma_c}{\beta/(\beta_a\beta_c)} \implies \gamma_b\beta_c=\gamma_c\beta_b
\]
Since $\gcd(\beta_b,\beta_c)=\gcd(\gamma_b,\gamma_c)=1$ we see that $\beta_b=\gamma_b$ and $\beta_c=\gamma_c$. 

It remains to show $\lambda_i=\mu_i$ for each $i\ge 2$. Substituting $h_i^{\prime}=c_ih_i$, $c_i\in k^*$, into the equations
\[
h_0^{\beta_0}+\lambda_ih_1^{\beta_1}+h_i^{\beta_i}=0 \quad {\rm and}\quad (h_0^{\prime})^{\beta_0}+\mu_i(h_1^{\prime})^{\beta_1}+(h_i^{\prime})^{\beta_i}=0
\]
shows 
$(\lambda_ic_0^{\beta_0}-\mu_ic_1^{\beta_1})h_1^{\beta_1}+(c_0^{\beta_0}-c_i^{\beta_i})h_i^{\beta_i}=0$.
Since $h_1$ and $h_i$ are algebraically
independent by {\it Proposition\,\ref{fact-closed}}, we see that $\mu_ic_1^{\beta_1}=\lambda_ic_0^{\beta_0}$ for each $i\ge 2$. In particular, $\lambda_2=\mu_2=1$ implies 
$c_1^{\beta_1}=c_0^{\beta_0}$. Therefore, $\mu_i=\lambda_i$ for each $i\ge 2$.

This completes the proof of the theorem.
\end{proof}

\subsection{Dimension Three}

In \cite{Ishida.77}, Ishida classified affine UFDs of dimension three over $k$ which admit an effective unmixed $\Z^2$-grading with $B_0=k$. 
By {\it Theorem\,\ref{classify}}, these rings are defined by trinomial data over $k$ using a partition of the form $n=1+\cdots +1+2$ for some integer $n\ge 2$. 
From this, we obtain the following description.

\begin{theorem}\label{dim-3} Let $B$ be a $k$-domain. The following conditions for $B$ are equivalent.
\begin{itemize}
\item [{\bf (a)}] $B$ is an affine UFD of dimension three over $k$ which admits an effective unmixed $\Z^2$-grading with $B_0=k$. 
\item [{\bf (b)}] There exists trinomial data $\Delta$ over $k$ in Mori form defined by either (1) $r=1$ and $\beta_0=\beta_1=1$ or (2) $r\ge 2$,  $\lambda_2=1$ and integers $\beta_0>\beta_1>\cdots >\beta_r\ge 2$, 
such that
\[
B\cong_kk[\Delta ][u,v]/(T_0^{\beta_0}+\lambda T_1^{\beta_1}+u^cv^d)
\]
where $\lambda\in k^*\setminus\{ \lambda_2,\hdots ,\lambda_r\}$, 
$(c,d)\in\N^2\setminus\{ (0,0\}$, $c\ge d$ and $\gcd (c,d,\beta_0\cdots\beta_r)=1$. 
\end{itemize}
\end{theorem}
Note that $A\cong k[\Delta ]^{[1]}$ if $(c,d)=(1,0)$.
The following is a proof that condition (a) implies condition (b). 
\begin{proof} Assume that the condition (a) holds. By {\it Theorem\,\ref{classify}}, there exist $m\geq0$ and trinomial data $\widetilde{\Delta}$ over $k$ such that
\[
B\cong_kk[\widetilde{\Delta}]^{[m]}=\left(k[T_0,\hdots ,T_{r+1}]/(T_0^{\beta_0}+\lambda_iT_1^{\beta_1}+T_i^{\beta_i})_{2\le i\le r+1}\right)^{[m]}
\]
If $m=1$, then ${\rm tr.deg}_kk[\widetilde{\Delta}]=2$, hence $\widetilde{\Delta}$ can be assumed to be in Mori form.

Assume that $m=0$. 
Since ${\rm tr.deg}_kB=3=n-r+2$, there exists $i\in\{0,\ldots,r+1\}$ such that $n_i=2$ and $n_j=1$ for any $j\not=i$. 

\medskip
{\it Case 1: $n_i=2$ for some $i\geq 3$}. By symmetry, we may assume that $n_{r+1}=2$. Therefore $T_{r+1}^{\beta_{r+1}}=u^cv^d$ for some $(c,d)\in\N^2\setminus\{(0,0)\}$ with $c\geq d$ and $\gcd(c,d,\beta_0\cdots\beta_r)=1$. 
Let $\Delta$ be the trinomial data defined by $(\beta_0,\ldots,\beta_{r})$ and $(\lambda_2,\ldots,\lambda_r)$. Since $\Delta$ is defined by the unit partition, we may assume that it is in Mori form. Therefore
\[
B\cong_kk[\Delta ][u,v]/(T_0^{\beta_0}+\lambda_{r+1}T_1^{\beta_1}+u^cv^d)
\]

\medskip
{\it Case 2: $n_0=2$}. Then $T_0^{\beta_0}=u^cv^d$ for some $(c,d)\in\N^2\setminus\{(0,0)\}$. By replacing $T_0^{\beta_0}$ by 
\[
-(\lambda_{r+1}T_1^{\beta_1}+T_{r+1}^{\beta_{r+1}})
\]
we have
\[
T_{r+1}^{\beta_{r+1}}+\mu_iT_1^{\beta_1}+(\xi_iT_i)^{\beta_i}=0 \:\: (2\leq i\leq {r}), \quad 
T_{r+1}^{\beta_{r+1}}+\lambda_{r+1}T_1^{\beta_1}+u^cv^d=0
\]
where $\mu_i=\lambda_{r+1}-\lambda_i$ and $\xi_i\in k$ with $\xi_i^{\beta_i}=-1$. 
Since $(\beta_{r+1},\beta_1,\ldots,\beta_r,\beta_0)$ and $(\mu_2,\ldots,\mu_r,\lambda_{r+1})$ give the trinomial data with the partition $n=1+\cdots+1+2$, the assertion follows from Case 1. 

\medskip
{\it Case 3: $n_1=2$}. Then $T_1^{\beta_1}=u^cv^d$ for some $(c,d)\in\N^2\setminus\{(0,0)\}$ and $u^cv^d+\lambda_i^{-1}T_0^{\beta_0}+(\xi_iT_i)^{\beta_i}=0$ for $2\leq i\leq r+1$, where $\xi\in k$ with $\xi^{\beta_i}=\lambda_i^{-1}$. This is in Case 2. 
\end{proof}

The ring $B$ in this theorem can also be described as follows. Let $A=k[T_0,T_1]\subset B$, noting that $A\cong k^{[2]}$. Define trinomial data $\Theta$ over $k$ by 
by the partition $3=1+2$, integer vectors $\{ (\beta_0),(\beta_1),(c,d)\}$ and $\lambda\in k^*$. Then:
\[
 k[\Theta ]=A[u,v]/(T_0^{\beta_0}+\lambda T_1^{\beta_1}+u^cv^d) \quad {\rm and}\quad B\cong k[\Delta ]\otimes_Ak[\Theta ] 
 \]

%%%%%%%%%%%%%%%%%%%%%%%%%%%%%%%%%%%%%%%%%%%%%%%%%%%%%%%%%%%%%%%%%%%%%%%%%%%%%%

\section{Application: Homogeneous LNDs of Polynomial Rings}\label{apps}

In this section, $k$ is an algebraically closed field of characteristic zero. 

\begin{theorem}\label{codim-2} Let $B$ be a unirational UFD of transcendence degree $n\ge 3$ over $k$ with $B^*=k^*$. Suppose that $B$ has an effective $G\cong\Z^{n-2}$-grading.
Let $D\in {\rm LND}(B)$ be nonzero and homogeneous, and let $A=\krn D$, which is a graded subring. Let $\Gamma\subset G$ be the subgroup generated by $\deg (A)$. 
\begin{itemize}
\item [{\bf (a)}] If ${\rm rank}(\Gamma )=n-2$ and $A_0=k$ then $A\cong k[\Delta ]^{[m]}$ for some trinomial data $\Delta$ over $k$ and $m\ge 0$, and $B$ is rational. 
\item [{\bf (b)}] If ${\rm rank}(\Gamma )\le n-3$ then $B\cong A^{[1]}$.
\end{itemize}
\end{theorem}

\begin{proof} Part (a). 
Since $A$ is the kernel of a nonzero locally nilpotent derivation of $B$, $A$ is a UFD of transcendence degree $n-1$ over $k$. In addition, since $B$ is unirational, $A$ is also unirational. Since ${\rm rank}(\Gamma)=n-2$, we see that
$\Gamma\cong\Z^{n-2}$. Therefore, $A$ admits an effective $\Z^{n-2}$-grading with $A_0=k$ and $A^*=k^*$. 
{\it Theorem\,\ref{main}} implies that $A\cong k[\Delta ]^{[m]}$ for some trinomial data $\Delta$ and $m\ge 0$. So $A$ is rational and, since ${\rm frac}(B)={\rm frac}(A)^{(1)}$, we see that $B$ is also rational. 

Part (b). 
Let the grading of $B$ be given by $\bigoplus_{g\in G}B_g$. Define the subring $B_{\Gamma}=\bigoplus_{g\in \Gamma}B_g$. By hypothesis, ${\rm rank}(\Gamma)\le n-3$. 
By \cite{Daigle.Freudenburg.Moser-Jauslin.17}, Proposition\,4.1, there exists homogeneous $p\in B$ such that $Dp\in A\setminus\{ 0\}$ and 
$B=B_{\Gamma}[p]$. If $p$ is algebraic over $B_{\Gamma}$, then there exists a homogeneous dependence relation
\[
\sum_{i=0}^Nb_ip^i=0 \quad (b_i\in B_{\Gamma}, N\ge 1, b_0, b_N\ne 0)
\]
i.e., $p$ and each $b_i$ are homogeneous. It follows (by homogeneity) that:
\[
\deg_{\mathfrak{g}}(b_Np^N)=\deg_{\mathfrak{g}}(b_0) \implies \deg_{\mathfrak{g}}(p^N)\in \Gamma
\]
But then
\[
G=\langle \Gamma,\deg_{\mathfrak{g}}p\rangle\cong\Z^{n-2} \quad {\rm and}\quad N\deg_{\mathfrak{g}}p\in \Gamma \implies \Gamma\cong\Z^{n-2}
\]
a contradiction. Therefore, $p$ is transcendental over $B_{\Gamma}$, $B=B_{\Gamma}[p]\cong B_{\Gamma}^{[1]}$ 
and $G=\Gamma\oplus\langle\deg_{\mathfrak{g}}p\rangle$, so $\Gamma\cong\Z^{n-3}$ and $\dim B_{\Gamma}=n-1$. 
Since $A$ is algebraically closed in $B$ and $A\subseteq B_{\Gamma}$,
it follows that $A=B_{\Gamma}$ and $B=A[p]\cong A^{[1]}$. 
\end{proof}

In case $B$ is a polynomial ring, we get the following. 

\begin{corollary}\label{poly-ring} Given the integer $n\ge 3$, let $B=k^{[n]}$ with an effective $G\cong \Z^{n-2}$-grading.
Let $D\in {\rm LND}(B)$ be nonzero and homogeneous with $A=\krn D$, and let $\Gamma\subset G$ be the subgroup generated by $\deg (A)$. 
Assume that $A_0=k$. 
\begin{itemize}
\item [{\bf (a)}] $A\cong k[\Delta ]^{[m]}$ for some trinomial data $\Delta$ over $k$ and $m\ge 0$. In particular, $A$ is affine and rational. 
\item [{\bf (b)}] If ${\rm rank}(\Gamma )\le n-3$ then $B\cong A^{[1]}$ and $A\cong_kk^{[n-1]}$. 
\end{itemize}
\end{corollary}

\begin{proof} By {\it Theorem\,\ref{codim-2}}, if ${\rm rank}(\Gamma )=n-2$ then $A\cong k[\Delta ]^{[m]}$ for some $m\ge 0$, and if ${\rm rank}(\Gamma )\le n-3$ then $B\cong A^{[1]}$. In the latter case, we see that $A$ is affine and ${\rm Spec}(A)$ is smooth. 
Since $A_0=k$ and $\dim A=n-1\ge 2$, it follows that $\Gamma\ne\{ 0\}$. By {\it Corollary\,\ref{smooth}}, $A\cong_kk^{[n-1]}$. 
\end{proof}

\begin{corollary}\label{dim-4} Let $B=k^{[4]}$ with an effective $\Z^2$-grading, let $D\in {\rm LND}(B)$ be nonzero and homogeneous, and let $A=\krn D$.
If $A_0=k$ then either $A\cong k^{[3]}$ or 
\[
A\cong k[x,y,u,v]/(x^a+y^b+u^cv^d)
\]
for $a,b,c,d\in\N\setminus\{ 0\}$ with $a> b\ge 2$ and $\gcd (a,b)=\gcd (c,d,ab)=1$.
\end{corollary}

\begin{proof} 
Note that $\dim A=3$. 
By {\it Corollary\,\ref{poly-ring}} and {\it Theorem\,\ref{dim-3}}, 
there exists trinomial data $\Delta$ over $k$ in Mori form defined by either (1) $r=1$ and $\beta_0=\beta_1=1$ or (2) $r\ge 2$,  $\lambda_2=1$ and integers $\beta_0>\beta_1>\cdots >\beta_r\ge 2$, 
such that
\[
A\cong_kk[\Delta ][u,v]/(T_0^{\beta_0}+\lambda T_1^{\beta_1}+u^cv^d)
\]
where $\lambda\in k^*\setminus\{ \lambda_2,\hdots ,\lambda_r\}$, $(c,d)\in\N^2\setminus\{ (0,0\}$, $c\ge d$ and $\gcd (c,d,\beta_0\cdots\beta_r)=1$. 
Suppose that $r\ge 2$. Then $T_0^{\beta_0}+T_1^{\beta_1}+T_2^{\beta_2}=0$ in $A$. By {\it Corollary\,\ref{tri-data}}, $T_0,T_1,T_2$ are irreducible and pairwise relatively prime in $A$.
Since $A$ is factorially closed in $B$, $T_0,T_1,T_2$ are irreducible and pairwise relatively prime in $B=k^{[4]}$. By {\it Lemma\,\ref{Iitaka}}, $\beta_2=1$, which gives a contradiction. Therefore, $r=1$ and $\beta_0=\beta_1=1$, so $k[\Delta ]=k[T_0,T_1]\cong k^{[2]}$
and $A\cong k[x,y,u,v]/(x^{\beta_0}+y^{\beta_1}+u^cv^d)$. 
\end{proof}

\begin{example} {\rm A {\bf monomial $k$-derivation} $D$ of $B=k[x_1,\hdots ,x_n]\cong k^{[n]}$ is one for which $Dx_i$ is a monomial in $x_1,\hdots ,x_n$ for each $i$. 
If $D$ is locally nilpotent, then it is necessarily triangular in $x_1,\hdots ,x_n$ up to re-ordering of the variables; see \cite{Freudenburg.17}, Proposition 3.4. 
These images can be homogeneous for various $\Z$-gradings of $B$. 
In particular, if $D$ is triangular and $Dx_1=\cdots =Dx_m=0$ for $1\le m<n$, then we are free to choose degrees for $x_1,\hdots ,x_m$ and for $D$. 
In this way, one obtains an effective unmixed $\Z^{m+1}$-grading of $B$ with $B_0=k$ for which $D$ is homogeneous. 
In case $n=4$, this gives a $\Z^2$-grading and {\it Corollary\,\ref{dim-4}} shows that either $\krn D\cong k^{[3]}$ or $\krn D\cong k[x,y,u,v]/(x^a+y^b+u^cv^d)$ for certain integers $a,b,c,d$. 
In \cite{Maubach.00}, Maubach proved by another method (without assuming $k$ is algebraically closed) that a triangular monomial derivation of $k^{[4]}$ has kernel isomorphic to $k^{[3]}$ or $k[x,y,u,v]/(x^a+y^b+u^cv)$, i.e., $d=1$.}
\end{example}

We also get a short proof of the following; see \cite{Freudenburg.17}, Theorem 5.5 and Corollary 5.6. 

\begin{theorem}\label{3d-ufd} 
Let $B$ be an affine UFD over $k$ where $\dim B=3$, $\dim ML(B)\le 1$ and $B^*=k^*$. Assume that $B$ has an effective $\Z$-grading, and let $D\in {\rm LND}(B)$ be homogeneous and nonzero with $A=\krn D$.
Assume that $A_0=k$. 
\begin{itemize} 
\item [{\bf (a)}] Either $A\cong_kk^{[2]}$ or $A\cong_kk[\Delta_0]$. 
\item [{\bf (b)}] If $B$ is a factorially closed subring of $k^{[m]}$ for some $m\ge 3$ then $A\cong_kk^{[2]}$. 
\item [{\bf (c)}] If $A\cong_kk^{[2]}$ then there exist homogeneous $f,g\in A$ such that $A=k[f,g]$. 
\item [{\bf (d)}] If $A\cong_kk[\Delta_0]$ then there exist homogeneous $x,y,z\in A$ such that:
\[
A=k[x,y,z] \quad {\rm and}\quad x^2+y^3+z^5=0
\]
\end{itemize}
\end{theorem}

\begin{proof} First observe that $A$ is affine by the Zariski Finiteness Theorem; see \cite{Freudenburg.17}, Section 6.4.1. 
Moreover, $A$ is a UFD and $\dim A=2$, since it is the kernel of a locally nilpotent derivation of a UFD of dimension three. 
Since $\dim ML(B)\le 1$ by hypothesis, we see that $A\not\subset ML(B)$. 
Also, the assumption that $A_0=k$ insures that the $\Z$-grading of $A$ is not trivial. So $A$ has an effective $\Z$-grading with $A_0=k$ and $A_0^*=k^*$. 
By {\it Theorem\,\ref{main}}, $A\cong_kk[\Delta ]$ for trinomial data $\Delta$ over $k$ with unit partition. 
By {\it Theorem\,\ref{stable}}, either $A\cong_kk^{[2]}$ or $A\cong_kk[\Delta_0]$. This proves part (a), and part (b) follows from
{\it Theorem\,\ref{stable}(c)}. 
Part (c) follows from \cite{Freudenburg.17}, Corollary 3.43, and part (d) follows from {\it Theorem\ref{max-torus}}.
\end{proof}

The next example is one in which the case $A\cong_kk[\Delta_0]$ in {\it Theorem\,\ref{3d-ufd}} occurs. For this example, the following lemma is needed. 

\begin{lemma}\label{Samuel} {\rm (\cite{Samuel.68}, Section 5, Part (3))} Let $B$ be a UFD over a field $K$ and let the finite group $G$ act on $B$ by algebraic automorphisms of $B$ as a $K$-algebra. 
If $B^*=K^*$ and ${\rm Hom}(G,K^*)$ is trivial then $B^G$ is a UFD.
\end{lemma} 

\begin{example} {\rm Let 
\[
\begin{pmatrix} x&y\cr z&w\end{pmatrix} \quad , \quad xw-yz=1
\]
be coordinates for $SL_2(k)$ and let $B$ denote the coordinate ring of $SL_2(k)$, which is a UFD with $B^*=k^*$. Let $B$ have the $\Z$-grading where $x,y,z,w$ are homogeneous and
$\deg (x,y,z,w)=(1,-1,1,-1)$. We define two group actions on $B$. 
\begin{enumerate}
\item 
Let $SL_2(k)$ act on itself by right multiplication. 
By restriction, $\G_a$ acts on $B$ by both the unipotent upper triangular subgroup and the unipotent lower triangular subgroup, where $B^{\G_a}=k[x,z]\cong k^{[2]}$ in the first case, and $B^{\G_a}=k[y,w]$ in the second case. 
Note that these are graded rings and $(B^{\G_a})_0=k$ in each case. 
\item 
Let $G\subset SL_2(k)$ be the binary icosahedral group, a group of order 120; see \cite{Riemenschneider.77} for a particular representation. 
Let $G$ act on $SL_2(k)$ by left multiplication. 
Since $B^*=k^*$, we have $(B^G)^*=k^*$. Let $\rho\in {\rm Hom}(G,k^*)$ be given. 
Since $k^*$ is abelian, $\rho$ factors through the quotient $G/[G,G]$ where $[G,G]$ is the commutator subgroup. It is well-known that $G$ is a perfect group, that is, $G=[G,G]$. 
Therefore, $\rho$ is trivial. By {\it Lemma\,\ref{Samuel}}, $B^G$ is a UFD. 
\end{enumerate}
It is clear that these two actions commute, which means that both the $SL_2(k)$-action on $B$ and the $\Z$-grading of $B$ restrict to $B^G$.
Therefore, $B^G$ has two homogeneous locally nilpotent derivations with kernels $k[x,z]^G$ and $k[y,w]^G$. It follows that:
\[
ML(B)\subseteq k[x,z]^G\cap k[y,w]^G=k \implies ML(B^G)=k
\]
For the first kernel we have $A:=k[x,z]^G=k[p,q,r]$ where $p^5+q^3+r^2=0$ and the degrees of $p,q,r$ in $x$ and $z$ are 12, 20 and 30, respectively; see \cite{Riemenschneider.77}.
So $A\cong_kk[\Delta_0]$. 
}
\end{example} 

We close with the following question, which arose in the course of our investigations. 

\begin{question} Let $G\cong\Z^n$ for some $n\ge 1$, let $M\subseteq G$ be a submonoid which generates $G$ as a group, and let $H\subseteq M$ be its maximal subgroup. 
Is $G/H$ torsion free? 
\end{question}

%%%%%%%%%%%%%%%%%%%%%%%%%%%%%%%%%%%%%%%%%%%%%%%%%%%%%%%%%%%%%%%%%%%%%%%%%%%%%%

%\bibliography{bibfile}
%\bibliographystyle{amsplain}

\bigskip

\noindent \address{Department of Mathematics\\
Western Michigan University\\
Kalamazoo, Michigan 49008}\\
USA\\
\email{gene.freudenburg@wmich.edu}
\bigskip

\noindent\address{National Institute of Technology (KOSEN)\\
Oyama College\\
771 Nakakuki, Oyama city, Tochigi 323-0806}\\
Japan\\
\email{t.nagamine14@oyama-ct.ac.jp}

\end{document}